\documentclass{article}
\usepackage{arxiv}

\usepackage{graphicx} % Required for inserting images
\usepackage{quiver}
\usepackage{bussproofs}
\usepackage{tikz-cd}
\usepackage{amsmath}
\usepackage{mathrsfs}
\usepackage{bm}
\usepackage{amsthm}
\usepackage{mathtools}
\usepackage{xcolor}
\usepackage[utf8]{inputenc}

\usepackage{hyperref}
\hypersetup{
    colorlinks,
    linkcolor={red!50!black},
    citecolor={blue!50!black},
    urlcolor={blue!80!black}
}
\usepackage[capitalize,nameinlink]{cleveref}
\Crefname{ex}{Example}{Examples}
\Crefname{pr}{Proof}{Proofs}

\theoremstyle{definition}
\newtheorem{definition}{Definition}[section]

\newtheorem{lemma}{Lemma}[section]

\newtheorem{proposition}{Proposition}[section]

\DeclareMathOperator{\parr}{\rotatebox[origin=c]{180}{\&}}

\newcommand{\B}{\mathfrak{B}}

\newcommand{\N}{\mathbb{N}}

\newcommand{\R}{\mathbb{R}}
\newcommand{\C}{\mathbb{C}}

\newcommand{\cI}{\mathcal{I}}
\newcommand{\cP}{\mathcal{P}}

\newcommand{\cQ}{\mathcal{Q}}

\newcommand{\cX}{\mathcal{X}}
\newcommand{\cY}{\mathcal{Y}}

\newcommand{\cat}[1]{\mathsf{#1}}

\newcommand{\dom}{\operatorname{dom}}

\newcommand{\cod}{\operatorname{cod}}
\newcommand{\id}{\operatorname{id}}

\newcommand{\RSR}{\operatorname{RSR}}
\newcommand{\im}{\operatorname{im}}
\newcommand{\bbrack}[1]{[\![ #1 ]\!]}
\DeclareMathAlphabet{\mathcal}{OMS}{cmsy}{m}{n}
\newcommand{\adj}{\underset{\leftarrow}{\overset{\rightarrow}{\scriptstyle \bot}}}

\date{}
\renewcommand{\headeright}{}
\renewcommand{\undertitle}{}

\title{A Category-theoretic Reconstruction of Logical Expressivism}
\author{ \href{https://orcid.org/0000-0002-9374-9138}{\includegraphics[scale=0.06]{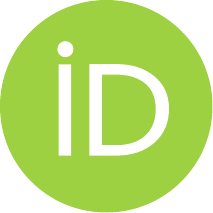}\hspace{1mm}Kristopher Brown} \\
	Topos Institute\\
	\texttt{kris@topos.institute} \\
}% \date{\vspace{-5ex}}

\begin{document}

\maketitle

% \tableofcontents

\vspace{-3mm}

\begin{abstract}
In the philosophical tradition of `analytic pragmatism', which attempts to account for linguistic meanings in terms of their practices of use, logical expressivism is a theory which offers a distinct perspective on logic. 
We shed light on Brandom and Hlobil's recent formalization of logical expressivism by reconstructing its formal semantics through the use of universal constructions. 
This reveals similarities with categorical logic: there is a focus on internalizing judgment structure, and connectives arise from adjunctions.
\end{abstract}

\noindent {\bf Notation:} We denote multisets of elements from some set $X$ with $\N[X]$, whereas the power set of $X$ is $\cP[X]$. Singletons and unions of multisets are, as with sets, denoted with $\{-\}$ and $\cup$ respectively.
We liberally apply the bijections $\N[X+X]\cong \N[X]\times \N[X]$ and $\cP[X+X]\cong\cP[X]\times \cP[X]$ implicitly. 
Functor compositions like $GF$ are shorthand for $G \circ F$, though functors (and morphisms in general) are sometimes composed in diagrammatic order, written $F\cdot G$. 
At times, pairs will be written $\langle -,-\rangle$ rather than $(-,-)$, especially if the elements inside contain further pairs. All mentioned monoids and quantales will be commutative and unital. The categories of quantales, monoids, and preordered monoids are respectively denoted $\cat{Quant}$, $\cat{CMon}$, and $\cat{PreOrdCMon}$. Given a preorder $\cX:=(X,\leq)$, let $\cX^\downarrow$ be the lattice of lower sets of $\cX$ ordered by $\subseteq$ whose {\it principal} lower sets are of the form $x^\downarrow:=\{x' \in X\ |\ x'\leq x\}$. Lastly, we warn that $\bot$ and $(-)^\bot$ are overloaded notations which have distinct but related meanings in the categories we will encounter. 

\vspace{-4mm}
\section{Background: inferentialism and logical expressivism}
\label{sec:back}

\vspace{-3mm}

In this section, we briefly summarize some concepts from {\it Reasons for Logic, Logic for Reasons}~\cite{hlobil2025reasons}. 
We can understand a domain of reasoning as having claims (that which can be asserted in that domain's vocabulary) and a notion of consequence, i.e. of some claims following from other claims. 
Example domains could range from `group theory' to `the game of chess' to `18th century maritime law'. 
A special class of examples comes from formal logics (e.g. classical, linear, paraconsistent) and their consequence relations. 
These consequence relations, despite their diversity, share some common structure.
For example, the consequence {\it operator}, which takes a set of premises to its set of consequences, is often assumed to satisfy the Tarskian conditions of reflexivity, monotonicity, and idempotence. 
From the proof-theoretic perspective, such constraints respectively correspond to structural rules of identity, weakening, and cut.
The formalization of a domain, i.e. the encoding of its claims within a formal system such that its consequence relation coincides with the logical consequence relation, is a powerful tool if successfully accomplished.
This is because weakening allows for portability of conclusions in different contexts, and cut allows for composability of reasoning. 
Thus it is a remarkable achievement that we have suitable logical encodings for some domains: for example, quantum logic for statements about  quantum-mechanical observables \cite{birkhoff1975logic}, intuitionistic logic for constructive mathematics \cite{dummett2000elements}, linear logic for resource-sensitive reasoning \cite{girard}. 

Inspired by the success of formalization of (certain aspects of) scientific reasoning, one could take the view that good reasoning always has this form: insofar as we genuinely mean something by `It is a dog' and `It is a mammal', and the former is a reason for the latter, then it must be the case that there must exist {\it some} encoding in {\it some} logic that is able to recover this inference as a logical consequence.
Russell's {\it On Denoting} \cite{russell1905denoting} offers an example of logical analysis of ordinary language expressions like ``the present King of France'' using classical predicate logic; however, one can reject his analysis on the basis of it {\it not} reproducing our ordinary inferences: e.g. he would be forced to endorse ``The present King of France is bald'' is a good reason for ``The moon is made of cheese'', as his analysis of the antecedent is that it is meaningful but false. There are many ways to react to this failure: one could seek a different encoding in classical logic, one can seek to develop new kinds of logics (e.g. free logics) to deal with `non-existent entities', and, lastly, we emphasize one could reject the attempt to ground the domain of interest in logical reasoning. It may be the case that legal, ethical, and medical reasoning do not satisfy the Tarskian constraints.\footnote{
Of the three constraints characteristic of logical consequence, weakening is the easiest to drop.
Categorical logic, by moving from cartesian closed categories to symmetric monoidal closed categories, allows for viewing weakening as a property rather than assumed structure; however, identity and cut remain fixed.
We note that rejection of cut is one way to make sense of logical paradoxes (such as the liar sentence)~\cite{ripley2013paradoxes}, and there many ordinary language (non-paradoxical) examples~\cite{simonelli2023bringing}.}

In contrast to this logico-semantic approach to meaning, {\it inferentialism} about meaning is the thesis that we should firstly think of the meaning of a sentence in terms of its good inferences and secondarily (if at all) think of it as represented by a formal term in some deductive calculus or an element in some semantic domain; this is the opposite order of explanation relative to the story explained above \cite{brandom1994making}. In this work, we focus on the formal treatment of inferentialism of \cite{hlobil2025reasons}, noting there exist other formal treatments, too \cite{incurvati2023reasoning,peregrin2014inferentialism}. We {\it begin} with some norm governing good judgments involving atomic sentences. 

\begin{definition}
\label{def:signedif}
A {\it signed incompatibility frame} is a set $X$ equipped with a subset $\bot\subseteq \N[X+X]$. {\it Idempotent} signed incompatibility frames are sets $X$ equipped with $\bot \subseteq \cP[X+X]$. Elements of $X$ are called {\it propositional atoms}. Elements of $\N[X+X]$ (or $\cP[X+X]$ in the idempotent case) are called {\it positions}, and $\bot$ is the subset of {\it incompatible} positions.
\end{definition}

We equivalently refer to a signed incompatibility frame as an {\it implication frame} because we often interpret a position $(\Gamma,\Delta)$ as a candidate implication $\Gamma \vdash \Delta$, with sets (or multisets) of atoms on each side of the turnstile \cite[Def. 62]{hlobil2025reasons}. Such a candidate implication may or may not obtain, according to the frame; the `good implications', i.e. the ones which obtain, are the elements of $\bot$.\footnote{Following \cite{restall2005multiple}, a {\it bilateralist} reading of the turnstile $\Gamma\vdash \Delta$ is that `It is normatively out-of-bounds' to {\it accept} everything in $\Gamma$ and {\it reject} everything in $\Delta$. Interpreting sequents in this way (where the left hand side represents accepted claims, and the right hand side represents simultaneously rejected claims) explains why an `incompatible' position (i.e. out-of-bounds) is a `good' implication ($\Delta$ follows if you cannot reject it while accepting $\Gamma$).} For example, $a,b\vdash a, c$ is notation for $(\{a,b\},\{a,c\})\in \bot$ and  $a,b\nvdash c$ is notation for $(\{a,b\},\{c\})\notin \bot$. Idempotent frames implicitly enforce contraction as a structural rule, and both frames implicitly enforce exchange.

% In that context, the turnstile $\Gamma \vdash \Delta$ is taken to mean ``Acceptance of $\Gamma$ and rejection of $\Delta$ is out of bounds.'' 
% Thus the structural rule of cut can be read as the enforcing that {\it all} contexts ($\Gamma$ and $\Delta$) which are incompatible with both assertion of $A$ and denial of $A$ (i.e. $\Gamma, A\vdash \Delta$ and $\Gamma\vdash A,\Delta$ both obtain) are themselves incompatible (i.e. $\Gamma \vdash \Delta$ obtains). 
% We would prefer to be able to articulate when excluded-middle-like property holds with logical vocabulary rather than implicitly demand it of all propositional atoms.

Thinking of an implication frame's turnstile as `raw, input data' leads to {\it logical expressivism}, which takes the concept of inferences being `simply' good (as in an implication frame) to be explanatorily-prior to the concept of logically-good inferences \cite{brandom2018logical}. This can be formalized via the sequent rules of non-monotonic, multi-succedent logic (NMMS, \cite{hlobil2018shop,kaplan2018nmms}), seen in~\cref{fig:seq}, which derive the goodness of propositions built from logical connectives. Adding logical connectives expands a vocabulary from $X$ to ${\rm BF}(X)$, the set of Boolean formulas on propositional atoms $X$. Because NMMS rules are bidirectional, any sequent which involves logically-complex propositions is tantamount to (a conjunction of) statements involving fewer logical connectives and, ultimately, to an assertion that some particular set of sequents is contained within $\bot$. We can understand the new, logically complex propositions as being descriptions of $\bot$. This is the idea of logical expressivism: that the characteristic function of logic is to express features of some antecedent, prelogical system of implications. Here `express' means to make explicit, i.e. internalize, features of $\bot$ that, were originally only describable in a metavocabulary.\footnote{In \cite[Ex. 1]{kaplan2018nmms}, it is shown that, in a radically substructural setting, the $\otimes$L and $\otimes$R rules of linear logic are {\it not} expressive in this sense: one can derive $p\otimes q \vdash p\otimes q$ from atomic sequents in two different ways, such that its assertion does not determinately tell us something about $\bot\subseteq \N[X+X]$. It can be derived starting from $(p\vdash p)$ and $(q\vdash q)$ {\it or} starting from $(p,q\vdash q)$ and $(\vdash p)$. } 

\begin{figure}

%\hspace{-2mm}
\begin{minipage}{0.17\linewidth}
\begin{center}\boxed{
\doubleLine
\RightLabel{$\neg$R}
\AxiomC{$\Gamma, a\vdash \Delta$}
\UnaryInfC{$\Gamma \vdash \neg a,\Delta$}
\DisplayProof}
\end{center}
\end{minipage}
\begin{minipage}{0.17\linewidth}
\begin{center}\boxed{
\doubleLine
\RightLabel{$\neg$L}
\AxiomC{$\Gamma\vdash a, \Delta$}
\UnaryInfC{$\Gamma, \neg a \vdash \Delta$}
\DisplayProof}
\end{center}
\end{minipage}
\begin{minipage}{0.25\linewidth}
\begin{center}\boxed{
\doubleLine
\RightLabel{$\wedge$R}
\insertBetweenHyps{\hskip -2pt}
\AxiomC{$\Gamma\vdash a,\Delta$}
\AxiomC{$\Gamma\vdash b,\Delta$}
\BinaryInfC{$\Gamma \vdash a\wedge b,\Delta $}
\DisplayProof}
\end{center}
\end{minipage}
\begin{minipage}{0.18\linewidth}
\begin{center}\boxed{
\doubleLine
\RightLabel{$\wedge$L}
\AxiomC{$\Gamma, a,b \vdash \Delta$}
\UnaryInfC{$\Gamma,a\wedge b \vdash \Delta $}
\DisplayProof}
\end{center}
\end{minipage}
\begin{minipage}{0.2\linewidth}
\begin{center}\boxed{
\doubleLine
\RightLabel{$\vee$R}
\AxiomC{$\Gamma\vdash a,b,\Delta$}
\UnaryInfC{$\Gamma \vdash a\vee b,\Delta $}
\DisplayProof}
\end{center}
\end{minipage}
\vspace{2mm}

\begin{minipage}{0.25\linewidth}
\begin{center}\boxed{
\doubleLine
\RightLabel{$\vee$L}
\insertBetweenHyps{\hskip -2pt}
\AxiomC{$\Gamma,a\vdash \Delta$}
\AxiomC{$\Gamma,b\vdash \Delta$}
\BinaryInfC{$\Gamma,a\vee b \vdash \Delta $}
\DisplayProof}
\end{center}
\end{minipage}
\begin{minipage}{0.37\linewidth}
\begin{center}\boxed{
\doubleLine
\RightLabel{$\wedge{\rm R^c}$}
\insertBetweenHyps{\hskip -2pt}
\AxiomC{$\Gamma\vdash a,\Delta$}
\AxiomC{$\Gamma\vdash b,\Delta$}
\AxiomC{$\Gamma\vdash a,b,\Delta$}
\TrinaryInfC{$\Gamma \vdash a\wedge b,\Delta $}
\DisplayProof}
\end{center}
\end{minipage}
\begin{minipage}{0.36\linewidth}
\begin{center}\boxed{
\doubleLine
\RightLabel{$\vee{\rm L^c}$}
\insertBetweenHyps{\hskip -2pt}
\AxiomC{$\Gamma, a\vdash \Delta$}
\AxiomC{$\Gamma,b\vdash \Delta$}
\AxiomC{$\Gamma,a,b\vdash\Delta$}
\TrinaryInfC{$\Gamma,  a\vee b \vdash\Delta $}
\DisplayProof}
\end{center}
\end{minipage}
\caption{Sequent rules for NMMS, with contractive (demarcated by superscript $^{\rm c}$) and non-contractive variants for $\wedge$R and $\vee$L rules. All other rules are considered both contractive and non-contractive.
}
\label{fig:seq}
\end{figure}

\begin{definition}
\label{def:elab}
{\it Logical elaboration} is the mapping, from implication frames to implication frames, which sends some $(X,\bot\subseteq \N[X+X])$ to $({\rm BF}(X),\bot'\subseteq \N[{\rm BF}(X)+{\rm BF}(X)])$, with $\bot'$ characterized via the noncontractive rules of \cref{fig:seq}. We can also logically elaborate idempotent implication frames $(X,\bot\subseteq \cP[X+X])$, sending them to ${({\rm BF}(X),\bot'\subseteq \cP[{\rm BF}(X)+{\rm BF}(X)])}$ using the contractive rules of \cref{fig:seq} for deriving $\bot'$. 
\end{definition}

The {\it implication space semantics} is defined relative to a choice of an implication frame, ${\cX:=(X,\bot)}$, with \cref{sec:exidem,sec:exnonidem} as examples. At its core is the `range of subjunctive robustness' operation~\cite[Def. 63]{hlobil2025reasons}, 
which sends $\{\Gamma\vdash \Delta\}$ to the set of contexts which would be good implications if $\Gamma$ were added to the premises and $\Delta$ to the conclusions. 

\begin{definition}
  \label{def:rsr}
  The {\it range of subjunctive robustness} function, $\RSR\colon {\cP[\N[X+X]]\to\cP[\N[X+X]]}$ (or ${\cP[\cP[X+X]]\to\cP[\cP[X+X]]}$ in the idempotent setting), sends a set of candidate implications, such as $\{(\Gamma_1,\Delta_1),...,(\Gamma_i,\Delta_i)\}$, to the set ${\{(\Theta,\Omega)\ |\ \forall i\colon (\Gamma_i\cup \Theta,\Delta_i\cup \Omega) \in \bot\}}$. We will abbreviate $\RSR(A)$ as $A^*$.
\end{definition}

This operation allows us to define implicational roles \cite[Def. 65]{hlobil2025reasons}.

\begin{definition}
  \label{def:rolcont}
  An {\it implicational role} is the range of subjunctive robustness for some set of candidate implications. For some background frame $\cX$, we denote the set of roles with $\R:=\im(\RSR)$. 
\end{definition}

$\cX$ induces `symjunction' and `adjunction' operations to combine roles in $\R$~\cite[Defs. 67,69]{hlobil2025reasons}.\footnote{No intended connection to the category-theoretic notion of adjunction.} 

\begin{definition}
  Given some background $\cX$, the {\it symjunction} of roles $A,B \in \R$ is $A \sqcap B:=(A\cup B)^{**}$. Their {\it adjunction} is $A\sqcup B:=\{(\Gamma\cup \Gamma',\ \Delta\cup\Delta')\ |\ (\Gamma,\Delta) \in A,\ (\Gamma',\Delta') \in B \}^{**}$.  
\end{definition}

Implicational roles allow us to define conceptual contents \cite[Def. 66]{hlobil2025reasons}.

\begin{definition}
  A {\it conceptual content} is a pair of roles, and we denote the set of conceptual contents with $\C:=\R^2$. Let $\bbrack{A}:=\langle \texttt{a}_+,\ \texttt{a}_-\rangle$ and $\bbrack{B}:=\langle \texttt{b}_+,\ \texttt{b}_-\rangle$ be arbitrary conceptual contents in $\C$. We use the notation $\bbrack{-}^+:=\pi_1(\bbrack{-})$ to pick out the {\it premisory role} of the conceptual content, and likewise  $\bbrack{-}^-:=\pi_2(\bbrack{-})$ picks out the {\it conclusory role}.
\end{definition}

$\cX$ also induces a variety of operations on conceptual contents~\cite[Defs. 70, 88]{hlobil2025reasons}.

\begin{definition}
  \label{def:connectives}
\begin{align*}
  \bbrack{\neg A}&:=\langle \texttt{a}_-,\texttt{a}_+\rangle & \bbrack{A \wedge B}&:=\langle \texttt{a}_+\sqcup \texttt{b}_+,\ \texttt{a}_- \sqcap \texttt{b}_- \sqcap (\texttt{a}_-\sqcup \texttt{b}_-)\rangle \\ 
  \bbrack{A\vee B}&:=\bbrack{\neg(\neg A\wedge \neg B)} & \bbrack{A\to B}&:=\bbrack{\neg A\vee B} \\ 
  \bbrack{A \otimes B} &:= \langle \texttt{a}_+\sqcup \texttt{b}_+,\ (\texttt{a}_-^* \sqcup \texttt{b}_-^*)^*\rangle & \bbrack{A \oplus B} &:= {\langle \texttt{a}_+ \sqcap \texttt{b}_+,\ (\texttt{a}_-^*\sqcap  \texttt{b}_-^*)^*\rangle} \\
  \bbrack{A\parr B}&:=\bbrack{\neg(\neg A \otimes \neg B)}& \bbrack{A\& B}&:=\bbrack{\neg(\neg A \oplus \neg B)}
\end{align*}

\end{definition}

$\cX$ induces a semantic consequence relation on $\C$~\cite[Def. 68]{hlobil2025reasons}. 

\begin{definition}
  \label{def:semcons}
  Given (multi-)sets of conceptual contents $\vec{\texttt{A}}=\sum_i \langle\texttt{a}_{i+},\ \texttt{a}_{i-}\rangle$ and $\vec{\texttt{B}}=\sum_j \langle\texttt{b}_{j+},\ \texttt{b}_{j-}\rangle$, we define the semantic consequence relation $\vec{\texttt{A}}\vDash\vec{\texttt{B}} := (\bigsqcup_{i} \texttt{a}_{i+})\sqcup (\bigsqcup_{j}\texttt{b}_{j-}) \subseteq \bot$.
\end{definition}

Lastly, $\cX$ induces interpretations for its propositional atoms.

\begin{definition}
  \label{def:basecase}
  For all atoms $a \in X$, let $\bbrack{a}:=\langle \{(\{a\},\varnothing)\}^{**},\ \{(\varnothing,\{a\})\}^{**}\rangle$.
\end{definition}

To summarize the implication space semantics: starting with a set of propositional atoms and a (radically substructural) consequence relation $\vdash$ from an implication frame $\cX$, we induce a space $\C$ of semantic values, a semantic consequence relation $\vDash$, and an interpretation function for logical composites of the atoms. These choices of semantic values and clauses are shown in \cite[Ch. 5]{hlobil2025reasons} to have many noteworthy logical properties: for example, we can understand many flavors of propositional logic (classical, linear, K3, LP, ST, TS) in terms of conditions on our starting $\cX$.

However, despite strong philosophical motivations, there are many mathematical choices in the above formalism which may seem unusual or non-obvious to a traditional logician. We aim to show that these choices are natural insofar as they can be reconstructed via universal constructions, with the implication space semantics arising as the unit of an adjunction. Proofs of lemmas used throughout the text can be found in \cref{sec:lemmas}.

{\bf Related work:} In \cite{corfield2023type}, Corfield applies the logical expressivist interpretation of introduction and elimination rules to constructions in dependent type theory. 
In \cite{yetter1990quantales}, Yetter constructs Girard quantales from phase spaces, and in \cite{rosenthal1990note} Rosenthal shows every Girard quantale can be obtained in this way. 
In \cite{mitchell2001monoid}, the authors explore relationships between ordered commutative monoids and Girard quantales, although they do not go as far as defining a category of phase spaces. The generation of a *-autonomous category from a symmetric monoidal category with a distinguished object to serve as the dualizing object bears a {\it prima facie} connection to the Chu construction \cite{barr1991autonomous}; however, this construction fails to be relevant to reconstructing logical expressivism because the monoidal product is quite different from what is needed. Lastly, there are key differences between the present approach and traditional categorical logic that prevent a simple characterization of how they are related. In the categorical logic following \cite{lambek1988introduction}, propositions are objects of a category and derivations $\Gamma \vdash A$ are morphisms. In contrast, our setting is proof-irrelevant, and whether sequents $\Gamma \vdash A$ obtain or not is simply data which may (but need not) satisfy properties such as transitivity that would be automatic if derivations were morphisms in a category.

\section{Unsigned logic}
\label{sec:unsigned}
\subsection{Incompatibility frames and free Girard quantales}
\label{sec:defps}

In this section, we define a category of incompatibility frames (i.e. {\it unsigned} incompatibility frames). We begin with a series of functors: $F^\leq\colon \cat{Set\to PreOrd}$ the discrete order, $\N[-]\colon \cat{PreOrd\to PreOrdCMon}$ the free commutative monoid on a preorder, and $F^\vee\colon \cat{PreOrdCMon\to Quant}$ Day convolution (free cocompletion) of a preordered commutative monoid. Applying these functors to a set $X$ yields $\cP[\N[X]]$, the power set quantale on the set of multisets on $X$. Note that the $\otimes$ of this quantale is $A\otimes B:=\{a+b\ |\ a\in A, b\in B\}$ using $+$ for multiset addition. We also start with a category of pointed quantales:

\begin{definition}
  \label{def:pq}
Let $\cat{PQ}$ have objects which are quantales $\cQ$ equipped with a distinguished point $\bot \in \cQ$. Morphisms $f\colon \cX\to \cY$ are quantale homomorphisms which weakly preserve the point: $f(\bot_\cX)\leq_\cY \bot_\cY$. 
\end{definition}

Let $U^\bot\colon \cat{PQ\to Quant}$ be the obvious forgetful functor. We then construct a pullback $\cat{Set \times_{Quant} PQ}$ of $F^\leq\cdot \N[-]\cdot F^\vee$ and $U^\bot$, with projections $\pi_{\cat{Set}}$ and $\pi_{\cat{PQ}}$.

\begin{definition}
  \label{def:if}
Let $\cat{uIF}:=\cat{Set \times_{Quant} PQ}$ be the {\it category of incompatibility frames}. Objects are sets $X$ with a distinguished $\bot \subseteq \N[X]$.\footnote{Although formally the objects of this pullback are pairs $(X,Q)$ with $X$ a set and $Q$ being a pointed quantale, everything about the quantale is fixed by the pullback condition save for its choice of a distinguished $\bot \in \cP(\N[X])$.} Morphisms $f\colon \cX\to \cY$ are functions which preserve $\bot$, i.e. for any multiset $\Gamma \in \N[X]\colon \Gamma \in \bot_\cX \implies \N[f](\Gamma) \in \bot_\cY$. 
\end{definition}

For the following lemma, we denote the cartesian lift \cite{jacobs1999categorical} of a morphism $f$ with $\overline{f}$.

  \begin{lemma}
  \label{lemma:pullbackprojadjoint}
  Let $F\colon \cat{A \rightarrow C}$ and $G\colon \cat{B\rightarrow C}$ be functors such that $F\dashv U$ (with unit $\eta$ and counit $\varepsilon$) and $G$ is a fibration. Then, the pullback projection $\pi_\cat{B}\colon \cat{A \times_C B\rightarrow B}$ has a right adjoint $R\colon  \cat{B\rightarrow A \times_C B}$ given by $R(b)\mapsto (UG(b), \dom(\overline{\varepsilon_{G(b)}}))$.
\end{lemma}

\[\begin{tikzcd}[cramped, column sep=8mm]
	{\mathsf{A\times_C B}} && {\mathsf{B}} \\
	\\
	{\mathsf{A}} && {\mathsf{C}}
	\arrow[""{name=0, anchor=center, inner sep=0}, "{\pi_\mathsf{B}}", curve={height=-6pt}, from=1-1, to=1-3]
	\arrow["{\pi_\mathsf{A}}"', from=1-1, to=3-1]
	\arrow["\lrcorner"{anchor=center, pos=0.125}, draw=none, from=1-1, to=3-3]
	\arrow[""{name=1, anchor=center, inner sep=0}, "{\exists R}", curve={height=-6pt}, dashed, from=1-3, to=1-1]
	\arrow["{G\ (\text{fibration})}", from=1-3, to=3-3]
	\arrow[""{name=2, anchor=center, inner sep=0}, "F", curve={height=-6pt}, from=3-1, to=3-3]
	\arrow[""{name=3, anchor=center, inner sep=0}, "U", curve={height=-6pt}, from=3-3, to=3-1]
	\arrow["\dashv"{anchor=center, rotate=-90}, draw=none, from=0, to=1]
	\arrow["\dashv"{anchor=center, rotate=-90}, draw=none, from=2, to=3]
\end{tikzcd}\]

\begin{lemma}
  \label{lemma:ufibration}

  The forgetful functor $U^\bot \colon \cat{PQ \rightarrow Quant}$ is a fibration, with the lift of a quantale morphism $f\colon X\to Y$ and a pointed quantale $(Y, \bot)$ is $(X,\bot')$ where  $\bot':=f^*(\bot) = \bigvee\{x\in X\ |\ f(x)\leq_Y \bot\}$. This is the largest element of $X$ which allows for $f$ to weakly preserve the point.
\end{lemma}

We note that $F^\leq\cdot \N[-]\cdot F^\vee$ is the composition of three left adjoints. By \cref{lemma:pullbackprojadjoint,lemma:ufibration} we have that the pullback projection $\pi_{\cat{PQ}}$ has a right adjoint $U_{\cat{PQ}}\colon\cat{PQ\to uIF}$. Concretely, this sends a pointed quantale $(Q,\leq,\otimes,\bot)$ to the incompatibility frame $(Q,\bot')$ where the multiset ${\Gamma \in \N[Q]}$ is an element of $\bot'$ iff $\bigotimes_{\gamma \in \Gamma} \gamma \leq \bot$. The unit $\eta$ sends $x\mapsto \{x\}$ and counit $\varepsilon$ sends $\vec \Gamma \subseteq\N[Q]\mapsto \bigvee_{\Gamma \in \vec \Gamma} \bigotimes_{\gamma \in \Gamma} \gamma$.

\begin{definition}
  \label{def:gq}
A Girard quantale is a thin *-autonomous cocomplete category, i.e. a quantale with a dualizing object,~$\bot$.\footnote{Not to be confused with the bottom element of the quantale's lattice.}  Let $\cat{GQ}$ have Girard quantales as objects and, for morphisms $\cX\to \cY$, quantale homomorphisms which weakly preserve $\bot$, i.e. $f(\bot_\cX)\leq_\cY \bot_\cY$. 
\end{definition}

In any pointed quantale we can compute $(-)^\bot:= - \multimap \bot = \bigvee_{\{a\ |\ a \otimes - \leq \bot\}} a$. The point in a Girard quantale requires that $q^{\bot\bot}=q$ for all $q \in Q$. An important difference between $\cat{PQ}$ morphisms and $\cat{GQ}$ morphisms is that the latter satisfy the following condition:

\begin{definition}
  A pointed quantale morphism $f\colon (X,\bot)\to(Y,\bot')$ satisfies {\it continuity} when every element $x$ satisfies  $f(x)^{\bot'\bot'} = f(x^{\bot\bot})^{\bot'\bot'}$. A $\cat{uIF}$ morphism $f$ satisfies continuity if $\pi_{\cat{PQ}}(f)$ is continuous.
\end{definition}

We call this condition {\it continuity} because, in topological terms, maps for which the image of the closure is contained in the closure of the image (see the second formulation in \cref{lemma:continuityequiv}, below) are called {\it continuous} maps \cite[Def 16.A.1]{vcech1966topological}.

\begin{lemma}
  \label{lemma:continuityequiv}
  For a pointed quantale morphism $f\colon (X,\leq,\otimes_X,\bot_\cX)$ $\to$ ${(Y,\leq_Y,\otimes_Y, \bot_\cY)}$, the following are equivalent:

      \begin{enumerate}
    \item $\forall x \in X\colon f(x)^{\bot_\cY\bot_\cY} = f(x^{\bot_\cX\bot_\cX})^{\bot_\cY\bot_\cY}$
    \item $\forall x\in X\colon f(x^{\bot_\cX\bot_\cX})\leq_Y f(x)^{\bot_\cY\bot_\cY}$ 
    \item $\forall a,b\in X\colon a^{\bot_\cX}\leq_X b^{\bot_\cX} \implies f(a)^{\bot_\cY} \leq_Y f(b)^{\bot_\cY}$
  \end{enumerate}
\end{lemma}

\begin{lemma}
  \label{lemma:gqrefl}
   By forgetting the fact that $\bot$ is dualizing, $\cat{GQ}$ is a reflective subcategory of $\cat{DQ_{cont}}$, the wide subcategory of ${\cat{DQ}}$ restricted to continuous morphisms. The reflector sends $(Q,\leq,\otimes,\bot)$ to $(Q',\vee',\otimes',\bot)$ where $Q':=\{q\in Q\ |\ q=q^{\bot\bot}\}$ with $a \vee' b:= (a\vee b)^{\bot\bot}$, $a \otimes' b := (a \otimes b)^{\bot\bot}$. The reflection morphism sends $q \mapsto q^{\bot\bot}$.
\end{lemma}

If we can restrict our earlier adjunction between $\cat{uIF}$ and $\cat{PQ}$ to one between wide subcategories $\cat{uIF_{cont}}$ and $\cat{PQ_{cont}}$, we can obtain a composite adjunction.

\[\begin{tikzcd}
	{\mathsf{uIF}} & {\mathsf{PQ}} & \\
	{\mathsf{uIF_{cont}}} & {\mathsf{PQ_{cont}}} & {\mathsf{GQ}}
	\arrow[""{name=0, anchor=center, inner sep=0}, "{\pi_{\mathsf{PQ}}}", curve={height=-6pt}, from=1-1, to=1-2]
	\arrow[""{name=1, anchor=center, inner sep=0}, "{U_{\mathsf{PQ}}}", curve={height=-6pt}, from=1-2, to=1-1]
	\arrow["{{\footnotesize \rm wide}}", tail, from=2-1, to=1-1]
	\arrow[""{name=2, anchor=center, inner sep=0}, "{}", curve={height=-6pt}, from=2-1, to=2-2]
	\arrow["{{\footnotesize \rm wide}}"', tail, from=2-2, to=1-2]
	\arrow[""{name=3, anchor=center, inner sep=0}, "{}", curve={height=-6pt}, from=2-2, to=2-1]
	\arrow[""{name=4, anchor=center, inner sep=0}, curve={height=-6pt}, from=2-2, to=2-3]
	\arrow[""{name=5, anchor=center, inner sep=0}, "{{\footnotesize \rm full}}", curve={height=-6pt}, tail, from=2-3, to=2-2]
	\arrow["\dashv"{anchor=center, rotate=-90}, draw=none, from=0, to=1]
	\arrow["\dashv"{anchor=center, rotate=-90}, draw=none, from=2, to=3]
	\arrow["\dashv"{anchor=center, rotate=-90}, draw=none, from=4, to=5]
\end{tikzcd}\]

\begin{lemma}
  \label{lemma:adjrest}
  The adjunction $\pi_{\cat{PQ}}\dashv U_{\cat{PQ}}$ between $\cat{uIF}$ and $\cat{PQ}$ restricts to one between $\cat{uIF_{cont}}$ and $\cat{PQ_{cont}}$. 
\end{lemma}

We denote the composite adjunction $F^{\otimes\oplus}\dashv U^{\otimes\oplus}$ with unit $\eta^{\otimes\oplus}_\cX\colon \cX\to U^{\otimes\oplus}F^{\otimes\oplus}(\cX)$ which sends $x \in X \mapsto \{x\}^{\bot\bot}$.

\subsection{Logical interpretation of incompatibility frames}
\label{sec:log}

Much like a category might not have products, it need not be the case, for some incompatibility frame $\cX:=(X,\bot)$, that there exists a function $\otimes\colon \N[X]\rightarrow X$ such that for all multisets $\Gamma \in \N[X]:\otimes(\Gamma)^\bot = \Gamma^\bot$.  
It also need not be the case there exists a function $\oplus\colon \cP[\N[X]]\rightarrow X$ such that for all sets of multisets $\{\Gamma_1,...,\Gamma_n\} \subseteq \N[X]: \oplus(\{\Gamma_1,...,\Gamma_n\})^\bot = \{\Gamma_1,...,\Gamma_n\}^\bot$. 
A logical expressivist would say that any such functions $\otimes$ and $\oplus$ {\it make conjunction} (resp. {\it disjunction}) {\it explicit} because the binary cases of these constraints are expressed by bidirectional sequent rules, quantified over all $\Gamma \in \N[X]$: 

\vspace{2mm}

\begin{minipage}{0.45\linewidth}
\begin{center}
\doubleLine
\AxiomC{$\Gamma, a,b \vdash $}
\RightLabel{$\otimes$}
\UnaryInfC{$\Gamma,a\otimes b \vdash $}
\DisplayProof
\end{center}
\end{minipage}
\begin{minipage}{0.45\linewidth}
  \begin{center}
\doubleLine
\AxiomC{$\Gamma, a \vdash $}
\AxiomC{$\Gamma,b \vdash $}
\RightLabel{$\oplus$}
\BinaryInfC{$\Gamma,a\oplus b \vdash $}
\DisplayProof
\end{center}
\end{minipage}
\vspace{2mm}

We interpret the adjunction $F^{\otimes\oplus}\dashv U^{\otimes\oplus}$ as the free addition of these structures to an incompatibility frame. 
One criterion of adequacy for any introduction of logical connectives via sequent rules is {\it conservativity}. 
This criterion states the propriety of judgments which do not feature the newly introduced connectives should be unchanged, with the introduction of the $tonk$ connective being the paradigmatic counterexample~\cite{prior1960runabout}. 
The judgments in a frame $U^{\otimes\oplus} F^{\otimes\oplus}(\cX)$ which do {\it not} feature $\otimes$ or $\oplus$ are those which are in the image of $\eta^{\otimes\oplus}_\cX$. We characterize this property in our setting below.

\vspace{2mm}

\begin{definition}
\label{def:conservative}
A morphism of incompatibility frames $f\colon \cX\rightarrow \cY$ is {\it conservative} if $\forall \Gamma \in \N[X]\colon \Gamma \in \bot_\cX \iff \N[f](\Gamma)\in\bot_\cY$. 
\end{definition}

\vspace{2mm}

\begin{proposition} 
  \label{prop:cons}
  The adjunction unit $\eta^{\otimes\oplus}_\cX$ is a conservative $\cat{IF}$ morphism for any $\cX \in \cat{uIF}$. 
\end{proposition}

\begin{proof}

Let $\Gamma:=\sum_i \gamma_i \in \N[X]$. We'll also distinguish $\otimes$ of the free Girard quantale $F^{\otimes\oplus}(\cX)$ from $\otimes^{\rm Mink}$ of the free pointed quantale (the difference being that $\otimes$ adds an outer $(-)^{\bot\bot}$ closure). 
The $\implies$ direction of conservativity follows from $\eta_\cX$ being a $\cat{uIF}$ morphism. 
To show the reverse direction, we start from the assumption that $\N[\eta^{\otimes\oplus}](\Gamma)=\sum_i \eta^{\otimes\oplus}(\gamma_i) = \sum_i \{\gamma_i\}^{\bot\bot}\in\bot'$. 
Unpacking the definition of $\bot'$, this means $\bigotimes_i \{\gamma_i\}^{\bot\bot} \subseteq \bot$ in the free Girard quantale. Because $(-)^{\bot\bot}$ is a quantic nucleus, we obtain  $\bigotimes_i \{\gamma_i\}^{\bot\bot}=(\overset{\rm Mink}{\bigotimes_i}\{\gamma_i\})^{\bot\bot}=\{\Gamma\}^{\bot\bot}\subseteq \bot$. 
Because $(-)^{\bot\bot}$ is expanding, $\{\Gamma\}\subseteq \bot$ and we get $\Gamma \in \bot$. 

\end{proof}

The story so far is incomplete, as we are only capable of representing single-sided sequents. However, the following section shows how we can proceed analogously to represent two-sided sequents.

%%%%%%%%%%%%%%%%%%%%%%%%%%%%%%%%%%%%%%%%%%%%%%%%%%%%%%%%%%%%
\section{Signed logic}
\label{sec:signed}
\subsection{A category of implication frames}
\label{sec:defips}

In this section, we recover the two-sidedness of sequents by defining a category of signed incompatibility frames (i.e. implication frames). In order to define a new functor $\cat{Set\to Quant}$, we begin with an observation about the doubling functor $D\colon \cat{PreOrdCMon\to PreOrdCMon}$ sending $(P,\leq,\otimes)\mapsto{(P^2,\leq^2,\otimes^2)}$, 

\begin{lemma}
  \label{lemma:}
  The doubling functor is self-adjoint: $D\dashv D$. It has unit $x\mapsto \langle (x,0),(0,x)\rangle $ and counit $\langle (a,b),(c,d)\rangle \mapsto a\otimes d$.
\end{lemma}

We use this to define a new $F'\colon F^\leq\cdot \N[-]\cdot D\cdot F^\vee \colon \cat{Set\to Quant}$, which sends $X$ to $\cP[\N[X]^2]$.

\begin{definition}
    \label{def:signedframe}
The category $\cat{IF}$ of {\it implication frames} is defined as the pullback of $U^\bot$ and $F'$. 
Concretely, its objects are the implication frames of \cref{def:signedif}, and its morphisms are functions $f\colon X\rightarrow Y$ such that for all $(\Gamma,\Delta) \in \N[X]^2$ we have $\langle \Gamma,\Delta\rangle  \in \bot_\cX\implies \langle \N[f](\Gamma),\N[f](\Delta)\rangle  \in \bot_\cY$.
\end{definition}

We can also characterize $\cat{IF}$ as the subcategory of $\cat{uIF}$ obtained by restricting objects to those with underlying sets of the form $X+X$ and morphisms of the form $f+f$. 
However, the pullback formulation makes it straightforward to derive a pointed quantale adjunction, as before. 
We apply \cref{lemma:pullbackprojadjoint} to once again obtain a right adjoint for the pullback projection $\cat{IF \to PQ}$. 
We again restrict to continuous morphisms (defining $\cat{IF_{cont}}$ as a wide subcategory of $\cat{IF}$) and compose this adjunction with the reflective subcategory adjunction to $\cat{GQ}$ to obtain functors $F^{\otimes\neg\oplus}\colon \cat{IF_{cont}\to GQ}$ and $U^{\otimes\neg\oplus}\colon \cat{GQ\to IF_{cont}}$ such that $F^{\otimes\neg\oplus}\dashv U^{\otimes\neg\oplus}$.
Concretely, $U^{\otimes\neg\oplus}$ sends the Girard quantale $(Q,\leq,\otimes,\bot)$ to the implication frame $(Q^2,\bot')$ where $\vec{\texttt{A}}=\sum_i \langle\texttt{a}_{i+},\ \texttt{a}_{i-}\rangle$ and $\vec{\texttt{B}}=\sum_j \langle\texttt{b}_{j+},\ \texttt{b}_{j-}\rangle$, is in $\bot'$ iff $(\bigotimes_{i} \texttt{a}_{i+})\otimes (\bigotimes_{j}\texttt{b}_{j-}) \leq \bot$. 
The unit $\eta^{\otimes\neg\oplus}$ sends an element $x \in X$ to the pair $\langle (\{x\},\varnothing)^{\bot\bot},(\varnothing,\{x\})^{\bot\bot}\rangle$, which is a pair of elements of the free Girard quantale of the implication frame.

Although general implication frames are very expressive in their capacity to represent completely unrestricted consequence relations, we will have special interest in ones which are reflexive. 

\begin{definition}
  \label{def:refl}
An element $x \in X$ of an implication frame $(X,\ \bot\subseteq \N[X+X])$ is {\it reflexive} if $(x,x)\in \bot$. If all elements are reflexive, we say the frame itself is reflexive. Let $\iota^{\rm r}\colon \cat{IF}^{\rm r}_{\cat{cont}}\rightarrowtail \cat{IF_{cont}}$ be the full subcategory of reflexive frames with continuous morphisms. 
\end{definition}

We would like to extend this adjunction to $\cat{IF_{\cat{cont}}^r}$; however, this is not as simple as precomposing $F^{\otimes\neg\oplus}\dashv U^{\otimes\neg\oplus}$ with an adjunction $\cat{IF^r_{cont} \adj IF_{cont}}$. Although we can restrict any implication frame to its subframe of reflexive elements, this does {\it not} define a functor $\cat{IF}_{\cat{cont}} \to \cat{IF^r_{\cat{cont}}}$.\footnote{The issue is that the continuity property of morphisms is not preserved when restricting the functions to reflexive elements.}

\begin{lemma}
\label{lemma:reflrestrict}
Let $U_\pm\colon \cat{GQ \to IF^r_{\cat{cont}}}$ send a Girard quantale $\cQ$ to the reflexive subframe of $U^{\otimes\neg\oplus}(\cQ)$ and send a $\cat{GQ}$ morphism $f$ to the restriction of $U^{\otimes\neg\oplus} f$ to reflexive elements. This is functorial and right adjoint to $F_\pm:=\iota^{\rm r}\cdot F^{\otimes\neg\oplus}$.
\end{lemma}

\subsection{Logical interpretation of implication frames}

Adding the $\otimes\colon \N[X+X]\rightarrow X$ and $\oplus\colon \cP[\N[X+X]]\rightarrow X$ structure to implication frames works analogously to \cref{sec:log}. With implication frames, we can furthermore ask if there exists a function ${\neg\colon X\to X}$ that assigns each bearer $x \in X$ a `dual' bearer, $\neg x$, satisfying the equations ${(a,0)^\bot=(0,\neg a)^\bot}$ and ${(\neg a,0)^\bot=(0,a)^\bot}$. This can be expressed as the sequent rule below. 
  
  \begin{center}
\doubleLine
\AxiomC{$\Gamma,a \vdash b, \Delta $}
\UnaryInfC{$\Gamma, \neg b \vdash \neg a, \Delta$}
\RightLabel{$\neg$}
\DisplayProof
\end{center}

  Such a function need not exist for an implication frame: for example, none of the four functions $\{a,b\}\to\{a,b\}$ has this property in \cref{sec:exidem}. Introducing the doubling functor allowed for $\eta^{\otimes\neg\oplus}$ to freely add this structure.

\vspace{2mm}

Now that we are working in a signed setting, we can recover concepts introduced in \cref{sec:back}.

\begin{lemma}
  \label{lemma:rsrequiv}
  The elements of the free Girard quantale $F^{\otimes\neg\oplus}(\cX)$ are `roles' $\R$ (\cref{def:rolcont}), and the $\RSR$ function (\cref{def:rsr}) restricted to roles is precisely the $(-)^\bot$ operation of this quantale.  
\end{lemma}

Let $\cX:=(X,\bot)$ and $\hat \cX:=U^{\otimes\neg\oplus}(F^{\otimes\neg\oplus}(\cX)) = (\hat X,\hat \bot)$. We write an arbitrary element of $\hat X$ as $\texttt{A} = \langle \texttt{a}_+,\ \texttt{a}_-\rangle$ where $\texttt{a}_\pm \in \cP[\N[X+X]]$ (in particular, a pair of $(-)^{\bot\bot}$ closed subsets), and we write a multiset of such elements as $\vec{\texttt{A}} = \sum_{\texttt{A}\in \hat X} \texttt{A}^{i_\texttt{A}}$ for $i_\texttt{A} \in \N$.
The unit $\eta_\cX^{\otimes\neg\oplus}\colon \cX\to \hat \cX$ is precisely the interpretation of atoms of \cref{def:basecase}, assigning $x \mapsto \bbrack{x}$.
 Because $\hat \cX$ is an implication frame, we have a notion of consequence among multisets of its elements. 
 $(\vec{\texttt{A}},\vec{\texttt{B}}) \in \hat \bot$, i.e. $\vec{\texttt{A}} \vDash \vec{\texttt{B}}$, holds whenever $(\bigotimes_\texttt{A} \texttt{a}_+^{i_\texttt{A}})\otimes (\bigotimes_\texttt{B} \texttt{b}_-^{i_\texttt{B}})\subseteq \bot$, which matches the semantic consequence relation of the implication space semantics~(\cref{def:semcons}). 

Conservativity in $\cat{IF_{cont}}$ is determined by first interpreting the morphisms in $\cat{IF}$ and then applying \Cref{def:conservative}. The unit $\eta_\cX^{\otimes\neg\oplus}$ is conservative by the same argument as \cref{prop:cons}. Consequently, any restriction of $\eta^{\otimes\neg\oplus}$ (including $\eta_\pm$ from $F_\pm\dashv U_\pm$ or $\eta^{\rm c}$ of \cref{sec:idem}) is conservative. 

Elements $\texttt{A}$ of $\hat \cX$ which satisfy $\texttt{A}\vDash \texttt{A}$ are those for which $\texttt{a}_+ \otimes \texttt{a}_- \subseteq \bot$, which can also be expressed as $\texttt{a}_+ \subseteq \texttt{a}_-^\bot$. When we work instead with $\eta_\pm$, i.e. restricting to reflexive implication frames, then the elements of $\hat \cX$ have the following natural operations which are closed under the subset of $\cQ^2$ restricted to reflexive elements: 

\begin{lemma}
  \label{lemma:twistrestrict}
  Let $\hat \cX := U_\pm(\cQ)$ for some Girard quantale $\cQ$.
  These elements (which are pairs of elements of $\cQ$ that satisfy reflexivity) are closed under the quantale operations of the {\it twisted quantale} $\cQ\times\cQ^{\rm op}$, i.e. $\otimes\times \parr$,\footnote{Note $x\parr y := (x^\bot\otimes y^\bot)^\bot$, and this operation is a monoidal product for the opposite Girard quantale.} and $\vee \times \wedge$.
\end{lemma}

We use these operations to define semantic clauses on the elements of $\hat \cX$ for MALL connectives.

\begin{proposition}
  \label{prop:valid}
  Let $\cX$ be an implication frame with $\eta^{\otimes\neg\oplus}_\cX\colon \cX\to \hat \cX$. The twisted quantale operations below validate the logical rules of linear logic (MALL), independent of the choice of $\cX$. Let $\langle \texttt{a}_+,\texttt{a}_-\rangle, \langle \texttt{b}_+,\texttt{b}_-\rangle$ be elements of $\hat \cX$.

  \vspace{-2mm}
\begin{align*}
  &\bbrack{A \otimes B}:=\langle \texttt{a}_+\otimes \texttt{b}_+,\texttt{a}_-\parr \texttt{b}_-\rangle&& &\bbrack{A \oplus B}:=\langle \texttt{a}_+\vee \texttt{b}_+,\texttt{a}_-\wedge \texttt{b}_-\rangle\\
   &\bbrack{A \parr B}:=\langle \texttt{a}_+\parr \texttt{b}_+,\texttt{a}_-\otimes \texttt{b}_-\rangle&&  &\bbrack{A \& B}:=\langle \texttt{a}_+\wedge \texttt{b}_+,\texttt{a}_-\vee \texttt{b}_-\rangle
\end{align*}

Moreover $\bbrack{\neg A}:=\langle \texttt{a}_-,\ \texttt{a}_+\rangle$, with the caveat that linear negation is typically written as $(-)^\bot$.
\end{proposition}

\begin{proof}
  
The quantale operation $\otimes$ in the validations to the right of each rule is represented via concatenation for space.
Invertible steps are denoted with $\iff$, whereas implications are denoted with $\implies$.

\vspace{5mm}\hspace{-5mm}
\begin{minipage}{0.17\linewidth}

\begin{center}  \boxed{
\doubleLine
\AxiomC{$\Gamma \vdash A, \Delta$}
\RightLabel{$\neg$L}
\UnaryInfC{$\Gamma,\neg A \vdash \Delta$}
\DisplayProof}
\end{center}
\end{minipage}
\begin{minipage}{0.30\linewidth}
  \begin{align*}
    &&\pi_2(\langle \texttt{a}_+,\texttt{a}_-\rangle) &\subseteq (\Gamma_+\Delta_-)^\bot \\ 
     {\scriptscriptstyle \iff\hspace{-3mm}}&& \texttt{a}_-  &\subseteq (\Gamma_+\Delta_-)^\bot \\
     {\scriptscriptstyle \iff\hspace{-3mm}}&& \pi_1(\langle \texttt{a}_-,\texttt{a}_+\rangle)  &\subseteq (\Gamma_+\Delta_-)^\bot \\
  \end{align*}
\end{minipage}\hspace{5mm}
\begin{minipage}{0.17\linewidth}
\begin{center}\boxed{
\doubleLine
\AxiomC{$\Gamma, A\vdash \Delta$}
\RightLabel{$\neg$R}
\UnaryInfC{$\Gamma \vdash \neg A,\Delta$}
\DisplayProof}
\end{center}
\end{minipage}
\begin{minipage}{0.3\linewidth}
  \begin{align*}
    &&\pi_1(\langle \texttt{a}_+,\texttt{a}_-\rangle) &\subseteq (\Gamma_+\Delta_-)^\bot \\ 
     {\scriptscriptstyle \iff\hspace{-3mm}}&& \texttt{a}_+ &\subseteq (\Gamma_+\Delta_-)^\bot \\
     {\scriptscriptstyle \iff\hspace{-3mm}}&& \pi_2(\langle \texttt{a}_-,\texttt{a}_+\rangle)  &\subseteq (\Gamma_+\Delta_-)^\bot \\
  \end{align*}
\end{minipage}

\vspace{-5mm}

\hspace{-8mm}
\begin{minipage}{0.23\linewidth}
\begin{center}\boxed{
  \doubleLine
\AxiomC{$\Gamma,A,B \vdash \Delta$}
\RightLabel{$\otimes$L}
\UnaryInfC{$\Gamma, A \otimes B\vdash \Delta$}
\DisplayProof}
\end{center}
\end{minipage}
\hspace{-6mm}
\begin{minipage}{0.23\linewidth}
  \begin{align*}
    \texttt{a}_+  \texttt{b}_+ &\subseteq (\Gamma_+\Delta_-)^\bot  \\
    \text{ (Holds }&\text{definitionally)}\\
  \end{align*}
  \end{minipage}
\begin{minipage}{0.25\linewidth}
\begin{center}\boxed{
\AxiomC{$\Gamma \vdash A,\Delta$}
\AxiomC{$\hspace{-6mm}\Theta \vdash B,\Omega$}
\RightLabel{$\otimes$R}
\BinaryInfC{$\Gamma, \Theta \vdash A \otimes B, \Delta,\Omega$}
\DisplayProof}
\end{center}
\end{minipage}
  \begin{minipage}{0.31\linewidth}
    \begin{align*}
    (\Gamma_+\Delta_-\subseteq \texttt{a}_-^\bot) &\wedge (\Theta_+\Omega_- \subseteq \texttt{b}_-^\bot) \\
      {\scriptscriptstyle \implies} \Gamma_+\Delta_- \Theta_+&\Omega_-  \subseteq \texttt{a}_-^\bot  \texttt{b}_-^\bot\\
     {\scriptscriptstyle \iff} \Gamma_+\Delta_- \Theta_+&\Omega_-  \subseteq (\texttt{a}_- \parr \texttt{b}_- )^\bot\\
    \end{align*}
\end{minipage}

\vspace{-8mm}\hspace{-6mm}
\begin{minipage}{0.25\linewidth}
\begin{center}\boxed{
\doubleLine
\AxiomC{$\Gamma, A \vdash \Delta$}
\AxiomC{$\hspace{-6mm}\Gamma, B \vdash \Delta$}
\RightLabel{$\oplus$L}
\BinaryInfC{$\Gamma, A \oplus B \vdash \Delta$}
\DisplayProof}
\end{center}
\end{minipage}
  \begin{minipage}{0.29\linewidth}
    \begin{align*}
      (\Gamma_+\Delta_- \subseteq a_+^\bot)&\wedge (\Gamma_+\Delta_- \subseteq \texttt{b}_+^\bot) \\
    {\scriptscriptstyle \iff} \Gamma_+\Delta_-  &\subseteq \texttt{a}_+^\bot \wedge \texttt{b}_+^\bot \\
   {\scriptscriptstyle \iff} \Gamma_+\Delta_-  &\subseteq (\texttt{a}_+ \vee \texttt{b}_+)^\bot \\
    \end{align*}
\end{minipage}
\begin{minipage}{0.20\linewidth}
\begin{center}\boxed{
\AxiomC{$\Gamma \vdash A,\Delta$}
\RightLabel{$\oplus$R}
\UnaryInfC{$\Gamma \vdash A \oplus B,\Delta$}
\DisplayProof}
\end{center}
\end{minipage}
  \begin{minipage}{0.23\linewidth}
    \begin{align*}
      \Gamma_+\Delta_- &\subseteq \texttt{a}_-^\bot \\
     {\scriptscriptstyle \implies} \Gamma_+\Delta_- &\subseteq \texttt{a}_-^\bot \vee \texttt{b}_-^\bot \\
     {\scriptscriptstyle \iff}  \Gamma_+\Delta_-  &\subseteq (\texttt{a}_- \wedge \texttt{b}_-)^\bot\\
    \end{align*}
\end{minipage}

$\parr$ and $\&$ need not be checked, as they are definable as De Morgan duals of $\otimes$ and $\oplus$, i.e. $\bbrack{A\parr B}:=\bbrack{\neg(\neg A \otimes \neg B)}$ and $\bbrack{A\& B}:=\bbrack{\neg(\neg A \oplus \neg B)}$. The multiplicative true $\bbrack{1}:=\langle \bot^{\bot},\ \bot \rangle$ and additive false $\bbrack{0}:=\langle 0,0^\bot \rangle$ also satisfy MALL logical rules, with the remaining nullary connectives also being defined by De Morgan duals.
\end{proof}

Note \cref{prop:valid} did not require the elements of $\hat \cX$ to satisfy reflexivity, even though semantic clauses arose from a search for operations which are closed on the reflexive elements of $Q^2$.

\begin{proposition}
  The formulas of \cref{prop:valid} match the semantic clauses for $\neg, \otimes, \oplus, \parr,$ and $\&$ from \cref{def:connectives}.
\end{proposition}

\begin{proof}
  This is immediate, modulo some differences in presentation. For $\bbrack{A\otimes B}^+$: the formula for $\otimes$ matches the $\sqcup$ operation. For $\bbrack{A\otimes B}^-$: the definition of $\parr$ was unfolded in \cref{sec:back}, and by \cref{lemma:rsrequiv} $\RSR$ corresponds to $(-)^\bot$. For $\bbrack{A\oplus B}^+$: the join operation of the quantale $F^{\otimes\neg\oplus}(\cX)$ is union followed by $(-)^{\bot\bot}$ closure, i.e. $\sqcap$. For $\bbrack{A\oplus B}^-$, observe that $x \wedge y = (x^\bot\vee y^\bot)^\bot$ for any elements $x,y$ in a Girard quantale. Again, $\parr$ and $\&$ do not need to be checked as they are De Morgan duals in both definitions.
\end{proof}

Even when we enforce the structural rule of reflexivity, the choice of $\bot$ in our starting frame $\cX$ is important for which instances of cut obtain among the atomic sequents.

\begin{center}
\AxiomC{$\Gamma \vdash A, \Delta$}
\AxiomC{$\hspace{-6mm}\Theta, A \vdash \Omega$}
\RightLabel{Cut}
\BinaryInfC{$\Gamma, \Theta\vdash \Delta,\Omega$}
\DisplayProof
\end{center}

The top sequent expresses $\Gamma_+\otimes \Delta_- \subseteq \texttt{a}_-^\bot$ for some $\Gamma,\Delta$ and that $\texttt{a}_+ \subseteq (\Theta_+\otimes \Omega_-)^\bot$ for some $\Theta,\Omega$. The bottom asserts $\Gamma_+\otimes \Delta_- \subseteq (\Theta_+\otimes \Omega_-)^\bot$. That this entailment holds for {\it all} $\Gamma,\Delta,\Theta,\Omega$ is for $\texttt{a}_-^\bot \subseteq \texttt{a}_+$. %, which is true iff $\texttt{a}_+^\bot\subseteq \texttt{a}_-$. 
In summary, reflexivity requires $\texttt{a}_+\subseteq \texttt{a}_-^\bot$ and cut requires $\texttt{a}_-^\bot\subseteq\texttt{a}_+$. 
This means that the structural rules of classical linear logic demand that $\texttt{a}_+=\texttt{a}_-^{\bot}$: the premisory and conclusory roles of a semantic value cannot be independently chosen.

\begin{definition}
  \label{def:super}
  Given a set of propositional variables $X$ and consequence relation $\vdash$ on the set of MALL formulas on $X$, we say $\vdash$ is {\it supralinear} if $\vdash_{\rm MALL}\ \subseteq\  \vdash$, where $\vdash_{\rm MALL}$ is the MALL provability relation. Analogously, a consequence relation on the set of Boolean formulas on $X$ is {\it supraclassical} if ${\vdash_{\rm classical}\ \subseteq\ \vdash}$.
\end{definition}

\begin{proposition}
  \label{prop:supralin}
  For any $\cX=(X,\bot) \in \cat{IF^r_{cont}}$ with $\eta_\pm\colon \cX\to \hat \cX$, the consequence relation of $\hat \cX$ (which may fail to satisfy cut) is supralinear, and the atomic sequents that it validates are precisely $\bot$.
\end{proposition}

\begin{proof}
  Suppose $\Gamma \vdash_{\rm MALL}\Delta$. By cut-elimination for MALL \cite{girard1987linear}, $\Gamma \vdash_{\rm MALL}\Delta$ has a cut-free proof. The base cases for a proof come from a nullary logical rule or a single identity rule, which holds in $\hat \cX$ in virtue of being a reflexive implication frame. Each remaining step in the proof is a logical rule of MALL, which by \cref{prop:valid} holds in $\hat \cX$. Therefore $\Gamma \vdash_\cX \Delta$. That the valid atomic sequents are precisely $\bot$ is a restatement that $\eta_\pm$ is conservative.
\end{proof}

\section{Idempotent logic}
\label{sec:idem}

In this section, we restrict the $F_\pm\dashv U_\pm$ adjunction and obtain supraclassical consequence relations (rather than supralinear) parameterized by a choice of implication frame.

\begin{definition}
  \label{def:containment}
An element $x$ of an implication frame $(X,\bot\subseteq \N[X+X])$ is {\it idempotent} if $(\{x\},\{\})^\bot=(\{x,x\},\{\})^\bot$ and $(\{\},\{x\})^\bot=(\{\},\{x,x\})^\bot$. 
Note that, multiplicity-invariant subsets of $\N[X+X]$ are in bijection with subsets of $\cP[X+X]$. 
Let a {\it containment implication frame} be an implication frame where all elements are idempotent and satisfy {\it containment}, i.e. $(\{x\},\{x\})^\bot=\cP[X+X]$.
We define $\iota^{\rm c}\colon \cat{{IF^c_{cont}}\rightarrowtail IF^{\rm r}_{cont}}$ to be the subcategory of containment frames.\footnote{Reflexive frames could be restricted to those for which all elements are idempotent and, separately, to those for which all elements satisfy containment. However, we focus on the subcategory $\cat{IF^c_{cont}}$ where both properties obtain.}
\end{definition}

Demanding containment requires $\bot$ to contain all sequents of the form $\Gamma, x \vdash x,\Delta$; this is beyond asking for reflexivity, as it asks for reflexive sequents to be indefeasibly good.\footnote{ 
In the bilateralist interpretation of the turnstile, these are all of the positions in which some proposition is being both asserted and denied in the same breath.
}
Although we can restrict any reflexive implication frame to its subframe of idempotent, containment elements, this is {\it not} a functor $\cat{IF_{cont}^{r}\to IF_{cont}^{c}}$. 
We will repeat the tactic of \cref{lemma:reflrestrict}; however, we need to restrict both the implication frame side and the Girard quantale side of the adjunction.

\begin{definition}
  Let $\cat{GQ^{ji}}$ be the full subcategory of $\cat{GQ}$ where we restrict to {\it join-idempotent Girard quantales}, which are those for which every element $q \in \cQ$ can be expressed as some join of idempotent elements of $\cQ$.
\end{definition}

Let $U_\pm^{\rm c}:\cat{GQ^{ji}\to IF^{c}_{cont}}$ send a Girard quantale $\cQ$ to the idempotent and containment-satisfying subframe of $U_\pm(\cQ)$, and let $F_\pm^{\rm c}\colon \cat{IF^c_{cont}\to GQ^{ji}}$ be the corestriction of $\iota^{\rm c}\cdot F_\pm$ to join-idempotent Girard quantales.

\begin{lemma}
\label{lemma:undercont}
$U_\pm^{\rm c}$ and $F_\pm^{\rm c}$ are well-defined as functors, with $F_\pm^{\rm c}\dashv U_\pm^{\rm c}$.
\end{lemma}

Given some $\cX:=(X,\bot\subseteq \cP[X+X])$ satisfying containment, the unit of $F^{\rm c}_\pm\dashv U^{\rm c}_\pm$ is a map from $\cX$ to a subframe of $(Q^2,\bot')$, where $Q$ is the underlying set of $\cQ$ the free Girard quantale on $\cX$ and $\bot'$ for which once again $(\vec{\texttt{A}},\vec{\texttt{B}}) \in \hat \bot$ whenever $(\bigotimes_\texttt{A} \texttt{a}_+^{i_\texttt{A}})\otimes (\bigotimes_\texttt{B} \texttt{b}_-^{i_\texttt{B}})\subseteq \bot$, . 
We now investigate the sorts of operations make sense on this subset of idempotent and containment-satisfying pairs of quantale elements. 
First we note that, while the join of two idempotent elements in $\cQ$ is not necessarily idempotent, they can be joined in the quantale of idempotents.

\begin{lemma}
\label{lemma:idemsubquant}
  Let $\cQ:=(Q,\otimes,I,\vee)$ be any quantale. There is an idempotent quantale $(Q_\otimes,\otimes,\tilde\vee)$ with elements $Q_\otimes:=\{q \in Q\ |\ q \otimes q = q\}$, the same $\otimes$ as $\cQ$, and (binary) joins are given by $x\ \tilde \vee\ y := x \vee y \vee (x\otimes y)$.\footnote{This is a special case of general joins: $\tilde \bigvee_{i\in I} x_i := \bigvee_{\varnothing \subset J \subseteq_{\rm fin} I} \bigotimes_{j\in J} x_j$.}
\end{lemma}

The swap operation preserves idempotency and containment satisfaction. 
Although $\cQ\times \cQ^{\rm op}$ provided a natural set of operations on pairs of reflexive quantale elements, the $\parr$ operator cannot be restricted to idempotent elements,\footnote{Note that idempotent elements are not closed under $(-)^\bot$. 
Also, $\bot$ itself might not be an idempotent element, hence $Q_\otimes$ is not naturally a Girard quantale.} so we must consider other natural structures on pairs of quantale elements. 

\begin{lemma}
  \label{lemma:mixrestrict}

  The elements of $\hat \cX = U^{\rm c}_\pm(\cQ)$ for some Girard quantale $\cQ$ are closed under the monoidal operation of the {\it mixed monoid} $\cQ_\otimes^\otimes \times \cQ_\otimes^\vee$, whose components $\cQ_\otimes^\otimes$ and $\cQ_\otimes^\vee$ respectively are the underlying $\otimes$ and join monoids of $\cQ_\otimes$. 
\end{lemma}

\vspace{1mm}
We can use the mixed monoid multiplication operation to define the semantic formula for $\wedge$, i.e. ${\bbrack{A \wedge B} := \langle \texttt{a}_+\otimes \texttt{b}_+,\  \texttt{a}_-\ \tilde\vee\ \texttt{b}_-\rangle}$. 
We can then use the negation operation, which is the swapping of premisory and conclusory roles as before in \cref{prop:valid}, to define disjunction as the De Morgan dual: ${\bbrack{\texttt{A} \vee \texttt{B}} := \bbrack{\neg(\neg \texttt{A} \wedge \neg\texttt{B})}}$, recovering the semantic clauses of \cref{sec:back} for implication space semantics of classical logic. 

\vspace{0mm}
\begin{proposition}
  \label{prop:supra}
  For any $\cX=(X,\bot) \in \cat{IF^c_{cont}}$ with $\eta^{\rm c}\colon \cX\to \hat \cX$, the consequence relation of $\hat \cX$ is supraclassical, and the atomic sequents that it validates are precisely $\bot$.
\end{proposition}

\begin{proof}
  \cite[Thm. 76]{hlobil2025reasons} shows that the (contractive) implication space semantics is sound and complete for (contractive) NMMS, and \cite[Prop 25]{hlobil2025reasons}, which shows that, with containment, NMMS is supraclassical. Thus we just need show that we have successfully reconstructed the contractive implication space semantics. Conceptual contents are recovered as elements of $\hat \cX$, as argued in \cref{lemma:rsrequiv}. The assignment of conceptual contents to propositional atoms (\cref{def:basecase}) is reproduced by the adjunction unit. The consequence relation on conceptual contents matches \cref{def:semcons}, and the connectives for $\neg, \wedge, \vee$ match as well.
  
\end{proof}

\section{Conclusion}

We have shown the implication space semantics arises naturally from an adjunction between categories of implication frames and Girard quantales. 
The unit of the adjunction interprets a base implication frame into a semantic implication frame, and the counit of the adjunction determines the consequence relation of that frame.
The semantic clauses of the MALL connectives arises from restricting to reflexive implication frames, whereas the clauses for classical logic arise from restricting to idempotent and containment-satisfying frames.  
These adjunctions are summarized in the diagram below.
\[\begin{tikzcd}
	{\mathsf{uIF_{cont}}} && {\mathsf{IF_{cont}}} && {\mathsf{IF^r_{cont}}} && {\mathsf{IF^{c}_{cont}}} \\
	{\mathsf{GQ}} && {\mathsf{GQ}} && {\mathsf{GQ}} && {\mathsf{GQ^{ji}}}
	\arrow[""{name=0, anchor=center, inner sep=0}, "{{{F^{\otimes\oplus}}}}"', shift left, curve={height=6pt}, from=1-1, to=2-1]
	\arrow[tail, from=1-3, to=1-1]
	\arrow[""{name=1, anchor=center, inner sep=0}, "{{{F^{\otimes\neg\oplus}}}}"', curve={height=6pt}, from=1-3, to=2-3]
	\arrow["{{{\iota^{\rm r}}}}"', tail, from=1-5, to=1-3]
	\arrow[""{name=2, anchor=center, inner sep=0}, "{F_\pm}"', curve={height=6pt}, from=1-5, to=2-5]
	\arrow["{{{\iota^{\rm c}}}}"', tail, from=1-7, to=1-5]
	\arrow[""{name=3, anchor=center, inner sep=0}, "{F^{\rm c}}"', curve={height=6pt}, from=1-7, to=2-7]
	\arrow[""{name=4, anchor=center, inner sep=0}, "{{{U^{\otimes\oplus}}}}"', shift right, curve={height=6pt}, from=2-1, to=1-1]
	\arrow[""{name=5, anchor=center, inner sep=0}, "{{{U^{\otimes\neg\oplus}}}}"', curve={height=6pt}, from=2-3, to=1-3]
	\arrow[""{name=6, anchor=center, inner sep=0}, "{{{U_\pm}}}"', curve={height=6pt}, from=2-5, to=1-5]
	\arrow[equals, from=2-5, to=2-3]
	\arrow[""{name=7, anchor=center, inner sep=0}, "{{{U^{\rm c}}}}"', curve={height=6pt}, from=2-7, to=1-7]
	\arrow[tail, from=2-7, to=2-5]
	\arrow["\dashv"{anchor=center}, draw=none, from=0, to=4]
	\arrow["\dashv"{anchor=center}, draw=none, from=1, to=5]
	\arrow["\dashv"{anchor=center}, draw=none, from=2, to=6]
	\arrow["\dashv"{anchor=center}, draw=none, from=3, to=7]
\end{tikzcd}\]

\vspace{2mm}

More specifically, the semantic clauses for linear logic formulas arise naturally as operations on pairs of quantale elements that satisfy reflexivity, while the classical logic formulas likewise arise on pairs of idempotent quantale elements that satisfy containment. 

In \cite{hlobil2025reasons}, the non-contractive and contractive implication space semantics are told as related but disjoint stories; however, we have demonstrated that the contractive implication space semantics, $F^{\rm c}\dashv U^{\rm c}$, is a restriction of non-contractive implication space semantics, $F_\pm\dashv U_\pm$. 
The difference between non-contractive variants of the $\wedge R$ and $\vee L$ NMMS rules and their contractive counterparts (which include a third top sequent) is explained by them sharing the form of a quantale join, where the extra top sequent comes from the joins being computed in the quantale restricted to idempotent elements. 
Another curious feature, the semantic consequence relation on pairs which completely ignores the negative half of data on the left and the positive half of data on the right, emerges naturally from the counit to the $D\dashv D$ self adjunction of the doubling functor for (preordered) commutative monoids: $((a,b),(c,d))\mapsto a\otimes d$.
We also showed how imposing the conditions of reflexivity and containment+idempotence on an implication frame respectively generate supralinear and supraclassical semantic consequence relations. These are examples of properties which a domain of reasoning (as modeled by an implication frame) must satisfy for a particular kind of logic to be applicable.

{\bf Software implementation:} ROLE.jl \cite{brown2026library} is a Julia library which implements the concepts in this paper, in particular the declaration of (idempotent) implication frames $\cX:=(X,\bot)$ and the evaluation of semantic consequence between elements of $\hat \cX$, specified by logically complex formulas. The tests in this repository are also a source of worked examples of implication frames and the implication space semantics which are too large to present in this current work.

{\bf Future work}: A formal statement relating between traditional categorical logic and logical expressivism would account for their various (dis)similarities. NMMS and implication space semantics are intended to generalize various forms of {\it propositional} logic; however, there are many directions of possible generalization.
Informally, we could say implication frames are enriched in boolean values, but enriching in sets (yielding proof-relevant, predicate logics) or real numbers (probabilistic frames) would be particularly valuable if they are to model social norms and practices that regulate allowable inferences. 
Framing the underlying data structures and operations categorically is a key first step to this. 
Furthermore, additional exploration is needed of the possible applications of (co)limits and other structures within the category of implication frames, as it is important for the computational scalability of ROLE.jl that frames which consist of largely disjoint vocabularies can be implicitly represented as colimits in $\cat{IF}$ or related categories, such as the Kleisli category of the adjunctions we've described. 

{\bf Special thanks:} I thank Kevin Carlson for providing frequent guidance on diverse topics and Lucy Horowitz for helpful conversations and ROLE.jl contributions. 
The Research on Logical Expressivism working group (including Robert Brandom, Ulf Hlobil, Ryan Simonelli, Rea Golan, Shuhei Shimamura) has also been helpful in articulating and explaining logical expressivism.

\bibliographystyle{splncs04}
\bibliography{bibliography}

\crefalias{section}{appendix}

\appendix

\section{Implication space semantics examples}
\label{sec:ex}

Note in these examples we use the notation developed in the main body of the text rather than the introduction. Therefore, rather than $\sqcup$ we use $\otimes$, and rather than $\sqcap$ we use $\vee$.\footnote{With the caveat that there is the classical propositional connective $\vee$ distinct from the join operation $\vee$ on implicational roles in $\R$, which are shown in the main text to correspond to elements of a Girard quantale.}

\subsection{Idempotent example}
\label{sec:exidem}

In this example, we construct a simple idempotent implication frame whose implication space semantic consequence relation is supraclassical while also exhibiting a failure of monotonicity.

Consider an idempotent implication frame (\cref{def:signedif}) $\cX:=(X,\bot_\B)$, where $X := \{a,b\}$ (therefore the set ${\cP[X+X]}$ of {\it positions} has elements $0$, $a^+$, $a^-$, $a^+b^-$ and so on) and let the {\it incompatible} positions ${\bot_\B}$ be: $\{a^-,\ a^-b^-,\ a^+a^-,\ a^+a^-b^-,\ b^+b^-,\ b^+a^-b^-,\ a^+b^+,\ a^+b^+a^-,\ a^+b^+b^-,\ a^+b^+a^-b^-\}$.
That is: inferences like $a \vdash a,b$ are declared by fiat to obtain by this frame, i.e. $\cX$ establishes a norm where it is out-of-bounds to reject the position $\{a,b\}$ while accepting the position $\{a\}$. As another example, the inference $\vdash a$ is declared to obtain, i.e. it is out-of-bounds to simply reject $\{a\}$ given no other context. Moreover, the remaining inferences do {\it not} hold according to the frame, e.g. $b\nvdash a$ and $\nvdash b$. This $\bot_\B$ data is also visualized below, where a $\checkmark$ represents an element in $\bot_\B$:

\begin{center}
\begin{tabular}{||c||c|c|c|c||}
\hline\hline $\bot_\B$    & 0                    & $a^-$                & $b^-$            & $a^-b^-$             \\
\hline\hline $0$      & $\times$         & ${\checkmark}$ & ${\times}$ & ${\checkmark}$ \\
\hline $a^+$    & ${\times}$     & $\checkmark$         & ${\times}$ & $\checkmark$         \\
\hline $b^+$    & ${\times}$     & ${\times}$     & $\checkmark$     & $\checkmark$         \\
\hline $a^+b^+$ & ${\checkmark}$ & $\checkmark$         & $\checkmark$     & $\checkmark$         \\
\hline\hline \end{tabular}
\end{center}

By logical elaboration (\cref{def:elab}), there is another implication frame $\cX':=({\rm BF}(X), \bot')$. To determine what $\bot'$ is based on $\bot_\B$, we can use the NMMS rules of \cref{fig:seq}. For example, the $\neg R$ rule tells us that $\vdash a,b,\neg a$ obtains in $\cX'$ because $\cX$ validates $a \vdash a,b$. Likewise, the $\wedge R^{\rm c}$ rule allows us to derive $a,b\vdash a\wedge b$ in $\cX'$ from the facts that, in $\cX$, we have $a,b \vdash a$ and $a,b\vdash b$ and $a,b\vdash a,b$.

We now show the computation of ranges of subjunctive robustness ($\RSR$, or $(-)^\bot$, \cref{def:rsr}) for the sixteen positions of $\cX$. Each $\RSR$ value is a subset of $\cP[X+X]$. 
For example, $(a^+)^\bot$ and $(a^-)^\bot$ are the sets of positions incompatible when $a^+$ (resp. $a^-$) is added in. These sets are  $X_\pm:=\top\setminus\cP[\{a^+,b^-\}]$ and $X_b:=\top \setminus \{b^+,b^+a^-\}$ respectively, where $\top:=\cP[X+X]$. We can visually depict these sets below.

\begin{minipage}{0.45\linewidth}
\begin{center}
\begin{tabular}{||c||c|c|c|c||}
\hline\hline $X_\pm$   & 0                    & $a^-$                & $b^-$            & $a^-b^-$             \\
\hline\hline $0$  &  $\times$  & 	$\checkmark$	 & 	$\times$	 &  	$\checkmark$	\\
\hline $a^+$     &   $\times$   &   $\checkmark$   &  $\times$		&   $\checkmark$    \\
\hline $b^+$     &   $\checkmark$   &  $\checkmark$    &   $\checkmark$   &     $\checkmark$     \\
\hline $a^+b^+$  &  $\checkmark$    &   $\checkmark$   &  $\checkmark$    &    $\checkmark$      \\
\hline\hline \end{tabular}
\end{center}
\end{minipage}
\begin{minipage}{0.45\linewidth}
\begin{center}
\begin{tabular}{||c||c|c|c|c||}
\hline\hline $X_b$   & 0                    & $a^-$                & $b^-$            & $a^-b^-$             \\
\hline\hline $0$  &  $\checkmark$  & 	$\checkmark$	 & 	$\checkmark$	 &  	$\checkmark$	\\
\hline $a^+$     &   $\checkmark$   &   $\checkmark$   &  $\checkmark$		&   $\checkmark$    \\
\hline $b^+$     &   $\times$   &  $\times$    &   $\checkmark$   &     $\checkmark$     \\
\hline $a^+b^+$  &  $\checkmark$    &   $\checkmark$   &  $\checkmark$    &    $\checkmark$      \\
\hline\hline \end{tabular}
\end{center}
\end{minipage}

We can also define $X_\mp :=\top\setminus\cP[\{a^-,b^+\}]$. 
The value of $(-)^\bot$ for each position in $\cX$ is below:

\begin{center}
\begin{tabular}{||c||c|c|c|c||}
\hline\hline $(-)^\bot$    & 0                    & $a^-$                & $b^-$            & $a^-b^-$             \\
\hline\hline $0$      & $\bot_\B$         & $X_b$ & $X_\pm$ & $\top$ \\
\hline $a^+$    & $X_\pm$     & $\top$         & $X_\pm$ & $\top$         \\
\hline $b^+$    & $X_\mp$     & $X_\mp$     & $\top$     & $\top$         \\
\hline $a^+b^+$ & $\top$ & $\top$         & $\top$     & $\top$     \\
\hline\hline   
\end{tabular}
\end{center}

Note $\RSR$ is defined over {\it sets} of positions of $\cX$, not just individual positions like above. However, we can derive the remaining $\RSR$ values by taking intersections of the five $\RSR$ values we have. This yields one more possible value, $X_\bot=\{a^+,b^+\}^\bot = X_\pm\cap X_\mp$. Being able to compute $\RSR$ on entire sets of positions allows us to apply the operation twice to original positions of $\cX$, shown below.

\vspace{2mm}

\begin{center}
\begin{tabular}{||c||c|c|c|c||}
\hline\hline $(-)^{\bot\bot}$    & 0                    & $a^-$                & $b^-$            & $a^-b^-$             \\
\hline\hline $0$      & $X_b$         & $\bot_\B$ & $X_\mp$ & $X_\bot$ \\
\hline $a^+$    & $X_\mp$     & $\bot_\B$         & $X_\mp$ & $X_\bot$         \\
\hline $b^+$    & $X_\pm$     & $X_\pm$     & $X_\bot$     & $X_\bot$         \\
\hline $a^+b^+$ & $X_\bot$ & $X_\bot$         & $X_\bot$     & $X_\bot$     \\
\hline\hline   
\end{tabular}
\end{center}

\vspace{2mm}

There are six possible values of $\RSR$, so $\R=\{X_\bot,\bot_\B,X_b,X_\pm,X_\mp,\top\}$. Their inclusion poset is drawn below. 

\vspace{2mm}

\[\begin{tikzcd}[cramped]
	& \top & \\
	{X_\pm} && {X_b} \\
	{\bot_\mathfrak{B}} && {X_\mp} \\
	& {X_\bot}
	\arrow[from=2-1, to=1-2]
	\arrow[from=2-3, to=1-2]
	\arrow[from=3-1, to=2-1]
	\arrow[from=3-1, to=2-3]
	\arrow[from=3-3, to=2-3]
	\arrow[from=4-2, to=3-1]
	\arrow[from=4-2, to=3-3]
\end{tikzcd}\]

\vspace{2mm}

The $\vee$ operation table (with unit $X_\bot$) is on the left, and the $\otimes$ operation table (with unit $0^{\bot\bot}=X_b$) is on the right:

\vspace{2mm}

\begin{minipage}{0.45\linewidth}
\begin{center}
\begin{tabular}{||c||c|c|c|c|c|c||}
\hline\hline $\vee$  & $X_b$   & $X_\bot$ & $\bot_\B$ & $X_\pm$ & $X_\mp$ & $\top$ \\
\hline\hline   
$X_b$  & $X_b$   & $X_b$ & $X_b$ & $\top$ & $X_b$ & $\top$ \\ \hline  % 1
$X_\bot$  & $X_b$   & $X_\bot$ & $\bot_\B$ & $X_\pm$ & $X_\mp$ & $\top$ \\ \hline  % 2
$\bot_\B$  & $X_b$   & $\bot_\B$ & $\bot_\B$ & $X_\pm$ & $X_b$ & $\top$ \\ \hline % 3
$X_\pm$  & $\top$   & $X_\pm$ & $X_\pm$ & $X_\pm$ & $\top$ & $\top$ \\ \hline  % 4
$X_\mp$  & $X_b$   & $X_\mp$ & $X_b$ & $\top$ & $X_\mp$ & $\top$ \\ \hline  % 5
$\top$  & $\top$   & $\top$ & $\top$ & $\top$ & $\top$ & $\top$ \\ \hline\hline % 6
\end{tabular}
\end{center}
\end{minipage}
\begin{minipage}{0.45\linewidth}
\begin{center}
\begin{tabular}{||c||c|c|c|c|c|c||}
\hline\hline $\otimes$  & $X_b$   & $X_\bot$ & $\bot_\B$ & $X_\pm$ & $X_\mp$ & $\top$ \\
\hline\hline   
$X_b$  & $X_b$   & $X_\bot$ & $\bot_\B$ & $X_\pm$ & $X_\mp$ & $\top$ \\ \hline  % 1
$X_\bot$  & $X_\bot$   & $X_\bot$ & $X_\bot$ & $X_\bot$ & $X_\bot$ & $X_\bot$ \\ \hline  % 2
$\bot_\B$  & $\bot_\B$   & $X_\bot$ & $\bot_\B$ & $X_\pm$ & $X_\bot$ & $X_\pm$ \\ \hline % 3
$X_\pm$  & $X_\pm$   & $X_\bot$ & $X_\pm$ & $X_\pm$ & $X_\bot$ & $X_\pm$ \\ \hline  % 4
$X_\mp$  & $X_\mp$   & $X_\bot$ & $X_\bot$ & $X_\bot$ & $X_\mp$ & $X_\mp$ \\ \hline  % 5
$\top$  & $\top$   & $X_\bot$ & $X_\pm$ & $X_\pm$ & $X_\mp$ & $\top$ \\ \hline\hline % 6
\end{tabular}
\end{center}
\end{minipage}

\vspace{2mm}

Recall $\C:=\R^2$, the space of semantic values for Boolean formulas built out of $X$. We have base cases of $\bbrack{a}=\langle (a^+)^{\bot\bot},(a^-)^{\bot\bot} \rangle=\langle X_\mp,\bot_\B \rangle$ and $\bbrack{b}=\langle X_\pm,X_\mp \rangle$. We can use these semantic values to demonstrate that $\Gamma\vdash \Delta$ in $\cX$ is true iff $\bbrack{\Gamma}\vDash \bbrack{\Delta}$. For example, one can check if $\bbrack{a}\vDash \bbrack{a},\bbrack{b}$, which computed by checking: 

\begin{align*}
  \pi_1(\langle X_\mp,\bot_\B \rangle) \otimes \pi_2(\langle X_\mp,\bot_\B \rangle) \otimes \pi_2(\langle X_\pm,X_\mp \rangle) &\subseteq \bot_\B \\ 
  X_\mp\otimes \bot_\B \otimes X_\mp  &\subseteq \bot_\B \\ 
 X_\bot &\subseteq \bot_\B
\end{align*}

We also can compute semantic values of complex expressions. For example, we reasoned syntactically (on the basis of $\wedge R^{\rm c}$) that $a,b\vdash a\wedge b$, so we can check that $\bbrack{a},\bbrack{b}\vDash\bbrack{a \wedge b}$ with the following computation:

\begin{align*}
  \pi_1(\bbrack{a}) \otimes \pi_1(\bbrack{b}) \otimes \pi_2(\bbrack{a \wedge b}) &\subseteq \bot_\B \\ 
  \pi_1(\langle X_\mp,\bot_\B \rangle) \otimes \pi_1(\langle X_\pm,X_\mp \rangle) \otimes \pi_2(\langle X_\mp \otimes X_\pm,\ \bot_\B\vee X_\mp \vee (\bot_\B\otimes X_\mp)\rangle) &\subseteq \bot_\B \\ 
  X_\mp \otimes X_\pm \otimes (\bot_\B\vee X_\mp \vee (\bot_\B\otimes X_\mp)) &\subseteq \bot_\B \\ 
  X_\mp \otimes X_\pm \otimes (\bot_\B\vee X_\mp \vee X_\bot) &\subseteq \bot_\B \\ 
  X_\mp \otimes X_\pm \otimes (\bot_\B\vee X_\mp) &\subseteq \bot_\B \\ 
  X_\mp \otimes X_\pm \otimes X_b &\subseteq \bot_\B \\ 
  X_\bot &\subseteq \bot_\B \\ 
\end{align*}

By \cref{prop:supra}, the $\vDash$ consequence relation is supraclassical (because in $\cX$ every position which has overlap between premises and consequences is in $\bot_\B$), so we also know a large class of entailments which hold without having to explicitly calculate them. However, even though it is supraclassical, it is not monotonic due to $\vDash \bbrack{a}$ and $\bbrack{b} \nvDash \bbrack{a}$.

\subsection{Non-idempotent example}
\label{sec:exnonidem}

In this example, we attempt to construct a simple implication frame whose implication space semantic consequence relation is supralinear while also exhibiting a failure of transitivity (and, therefore, a failure of cut).

Let $X=\{\bullet\}$. Multisets of $X$ can be identified with natural numbers, thus the set of positions is $\N^2$. 
For brevity we'll denote pairs with concatenation, e.g. $(m,n)$ is written $mn$. 
Let $R=\{ii\ |\ i \in \N\}$ and define $\bot_\B=\{01,12\}\cup R$. 
Our implication frame of interest is then $\cX:=(X,\bot_\B)$. That is, $\cX$ endorses the implications $\vdash \bullet$ and $\bullet\vdash \bullet,\bullet$ (which could also be written $0\vdash 1$ and $1 \vdash 2$ respectively) as well as all implications which have the same number of $\bullet$ on each side.

We abbreviate $\{(m+i,n+i)\ |\ (m,n)\in\N^2\}$ as $mnR$. 
We also define an operation on $\N^2$ sending pairs to the smallest pair that has the opposite difference: $\overline{nm}=(n',m')$ where $m'=n-m$ and $n'=0$ if $m\leq n$, otherwise $n'=m-n$ and $m'=0$. 
In the following table, the $i$'th row and $j$'th column contains the value of $ij^\bot$.

\begin{center}
\begin{tabular}{||c||c|c|c|c||}
\hline\hline $(-)^\bot$    & 0                    & 1                & 2            & $m>2$             \\
\hline\hline $0$      & $\{01,12\}\cup R$         & $\{00,11\}\cup 10R$ & $\{10\}\cup 20R$ & $m0R$ \\
\hline 1    & $\{02\}\cup 01R$     & $\{01\}\cup R$         & $\{00\}\cup 10R$ & $(m-1)0R$         \\
\hline $n>1$    & $0nR$     & $0(n-1)R$     & $0(n-2)R$     & $\overline{nm}R$         \\
\hline\hline   
\end{tabular}
\end{center}

Note that $nmR^\bot=\overline{nm}R$ (as a consequence, $R^\bot=R$).  We now compute the $(-)^{\bot\bot}$ values:

\begin{itemize}
  \item $00^{\bot\bot}=(01^\bot\cap 12^\bot) \cap R^\bot=(\{00\}\cup 10R)\cap R=\{00\}$
  \item $01^{\bot\bot}=(00^\bot\cap 11^\bot) \cap 10R^\bot=11^\bot \cap 10R^\bot=(\{01\}\cup R)\cap \overline{10}R=(\{01\}\cup R)\cap01R=\{01\}$
  \item $02^{\bot\bot}=10^\bot\cap 20R^\bot = (\{02\}\cup 01R)\cap\overline{20}R=\{02\}\cap  02R = \{02\}$
  \item $10^{\bot\bot}=02^\bot\cap 01R^\bot = (\{10\}\cup 20R)\cap\overline{01}R=\{10\}\cap  10R = \{10\}$
  \item $11^{\bot\bot}=01^\bot\cap R^\bot = (\{00,11\}\cup 10R)\cap R=\{00,11\}\cap  R = \{00, 11\}$
  \item $12^{\bot\bot}=00^\bot\cap 10R^\bot = (\{01,12\}\cup R)\cap \overline{10}R=\{01,12\}\cap  01R = \{01, 12\}$
\end{itemize}

These can be organized in the following table: 

\begin{center}
\begin{tabular}{||c||c|c|c|c||}
\hline\hline $(-)^{\bot\bot}$    & 0                    & 1                & 2            & $m>2$             \\
\hline\hline $0$      & $\{00\}$         & $\{01\}$  & $\{02\}$ & $0mR$ \\
\hline 1    &   $\{10\}$  &   $\{00, 11\}$      &$\{01, 12\}$ &  $0(m-1)R$     \\
\hline $n>1$    &   $n0R$  &    $(n-1)0R$  &  $(n-2)0R$    &    $\overline{mn}R$      \\
\hline\hline   
\end{tabular}
\end{center}

There are too many possible values of $(-)^\bot$ for us to consider explicitly enumerating $\R$ and showing the tables for $\otimes$ and $\vee$. We'll just note that $mnR \otimes xyR = (m+x)(n+y)R$.

We now consider the implication space semantics for $\cX$, and we can check that $n\vdash m \iff \bbrack{n}\vDash \bbrack{m}$. The semantic value of $\bbrack{\bullet}$ is $\langle 10^{\bot\bot},01^{\bot\bot}\rangle = \langle \{10\},\{01\}\rangle$.

\begin{itemize}
  \item $\vdash \bullet$ in $\cX$, so we need to check $\vDash \bbrack{\bullet}$. To see that this obtains, note $\{01\}\subseteq \bot_\B$.
  \item $\bullet\vdash \bullet,\bullet$ in $\cX$, so we check $\bbrack{\bullet}\vDash\bbrack{\bullet},\bbrack{\bullet}$. To show this, we observe $(\{10\} \otimes \{01\}\otimes \{01\})^{\bot\bot}=\{12\}^{\bot\bot}=\{01,12\}\subseteq \bot_\B$.
  \item $\nvdash \bullet,\bullet$ in $\cX$, which matches  $\nvDash \bbrack{\bullet},\bbrack{\bullet}$ because $\{02\} \nsubseteq \bot_\B$.
  \item $\bullet,\bullet \nvdash \bullet$ in $\cX$. We see $\bbrack{\bullet},\bbrack{\bullet}\nvDash \bbrack{\bullet}$ because ${(\{10\}\otimes\{10\}\otimes\{01\})^{\bot\bot}=\{21\}^{\bot\bot}=10R\nsubseteq \bot_\B}$.
\end{itemize}

Now we consider computing semantic values for MALL formulas built from our single atomic proposition $\bullet$. Linear negation is swapping premisory and conclusory roles, so $\bbrack{\bullet^\bot}=\langle\{01\},\{10\}\rangle$. Then $\bbrack{\bullet \multimap \bullet}=\bbrack{\neg \bullet \parr \bullet}$ is equal to $\langle \{01\}\parr \{10\},\ (\{10\}\otimes \{01\})^{\bot\bot} \rangle=\langle\varnothing,\ \{00,11\}\rangle$. We compute the premisory role below:

\vspace{-3mm}
\begin{align*}
  \{01\}\parr \{10\}\\
  (\{01\}^\bot \otimes \{10\}^\bot)^\bot \\ 
  ((\{00,11\}\cup 10R )\otimes (\{02\}\cup 01R))^\bot \\
  (\{02,13\}\cup... \cup 11R)^\bot \\
  \{02,13\}^\bot\cap ... \cap \overline{11}R  \\
  02^\bot \cap 13^\bot \cap ... \cap R  \\
  (\{10\}\cup 20R) \cap 20R \cap ... \cap R\\
  20R \cap ... \cap R\\
  \varnothing
\end{align*}

To check whether linear modus ponens is valid, we test $\bbrack{\bullet},\bbrack{\bullet\multimap\bullet}\vDash\bbrack{\bullet}$, which holds because $\{10\}\otimes \varnothing \otimes \{01\} = \varnothing\subseteq \bot_\B$. 
This also could have been determined by \cref{prop:supralin}, since $\vDash$ is supralinear. However, note $\vDash$ is not transitive, as we have $\vDash \bbrack{\bullet}$ and $\bbrack{\bullet}\vDash \bbrack{\bullet},\bbrack{\bullet}$, yet we also have $\nvDash \bbrack{\bullet},\bbrack{\bullet}$.

\section{Lemmas}
\label{sec:lemmas}

\subsection*{\cref{lemma:pullbackprojadjoint}}

{\bf Lemma}
  Let $F\colon \cat{A \rightarrow C}$ and $G\colon \cat{B\rightarrow C}$ be functors such that $F\dashv U$ (with unit $\eta$ and counit $\varepsilon$) and $G$ is a fibration. Then, the pullback projection $\pi_\cat{B}\colon \cat{A \times_C B\rightarrow B}$ has a right adjoint $R\colon  \cat{B\rightarrow A \times_C B}$ given by $R(b)\mapsto (UG(b), \dom(\overline{\varepsilon_{G(b)}}))$.

\begin{proof}
  
Every pullback projection factors through a full subcategory inclusion into the comma category. We show the overall adjunction by first showing the $B$ is a reflective subcategory of the comma category and then that the pullback is a coreflective subcategory of the comma category: $\cat{A\times_C B}\ \ \underset{\leftarrow}{\overset{\rightarrowtail}{\scriptstyle \bot}} \ \ F\downarrow G\ \  \underset{\leftarrowtail}{\overset{\rightarrow}{\scriptstyle \bot}} \ \ \cat{B}$.

For the comma projection to be a reflector for the subcategory $b\mapsto (UG(b), b, \varepsilon_{G(b)})$, we consider the hom-set $(F\downarrow G)(\langle a,b',\phi\colon Fa\to Gb'\rangle ,\langle UG(b),b,\varepsilon_{G(b)}\rangle)$.  For this set of pairs $\{(h\colon a\to UG(b),\ k\colon b'\to b)\ |\ \phi\cdot G(k) = F(h)\cdot \varepsilon_{G(b)}\}$ to be naturally isomorphic with $\cat{B}(b',b)$, we need to show that there is exactly one $h$ that makes the equation true. This follows from the formulation of adjunctions \cite[Sec IV.1]{mac1971categories} which says, $(UG(b),\_\ ,\varepsilon_{G(b)})$ is the terminal object of $F \downarrow G(b)$ (overloading $G(b)$ as a constant functor $1\to\cat{C}$: this means for every $m\colon F(a)\to G(b)$ there is exactly one $h\colon a\to UG(b)$ such that $m=F(h)\cdot \varepsilon$): for every choice of $k$, $\phi\cdot Gk\colon F(a)\to G(b)$ is precisely that $m$. This bijection is natural because $(h,k)\mapsto k$ is projection, which is natural.

The coreflection requires $(F\downarrow G)(\langle a',b',1_{F(a')=G(b')}\rangle,\ \langle a,b,\phi\colon F(a)\to G(b)\rangle)$ be in natural bijection with $(\cat{A\times_C B})(\langle a',b'\rangle,\ \langle a,\dom\overline{\phi}\rangle)$. The first set is $\{\langle h\colon a'\to a,\ k\colon b'\to b\rangle\ |\ G(k)=F(h)\cdot \phi\}$. The second set is $\{\langle h\colon a'\to a,\ k'\colon b'\to \dom \overline{\phi}\rangle\ |\ F(h)=G(k')\}$. We need to show that, for any choice of $h\colon a'\to a$, the valid $k$ morphisms are in bijection with $k'$ morphisms. 

Start with $k$ such that $G(k)=F(h)\cdot \phi$. Because $F(a')=G(b')$ and $F(a)=\dom(\phi)$, we can say $F(h)\in \cat{C}(G(b'),\dom(\phi))$.  The universal property of $\overline{\phi}$ is that any $\cat{C}$ map into $\dom(\phi)$, such as $F(h)$, along with $\cat{B}$ map into $\cod(\overline{\phi})$, such as $k$, determine a unique lift: call this $!(k,F(h)):=k'\colon b'\to \dom(\overline{\phi})$, satisfying $k'\cdot \overline{\phi}=k$ and $G(k')=F(h)$. This second condition is precisely what is needed to be a member of the second hom set. 

In the reverse direction, we have a $k'\colon b'\to \dom(\overline{\phi})$ and assume $F(h)=G(k')$. Let $k=k'\cdot \overline{\phi}$. This is a morphism $b'\to b$. When we apply $G$ to it, we get $G(k')\cdot \phi$, which equals $F(h)\cdot \phi$ which is precisely the first hom set's condition.

These constructions are mutually inverse: $k\mapsto!(k,F(h))\mapsto !(k,F(h))\cdot\overline{\phi}=k$ (this is the {\it first} condition stated above for the lift). In the other direction: $k'\mapsto k'\cdot \overline{\phi} \mapsto !(k'\cdot \overline{\phi}, F(h))=k'$ (from the uniqueness of the lift: $k'$ also satisfies the properties of composing with $\phi$ to give $k$ and being sent to $F(h)$ by $G$). Moreover the bijection is natural because $k'\to k$ is just composition with $\overline{\phi}$, which is natural.
\end{proof}

\subsection*{\cref{lemma:ufibration}}

{\bf Lemma}
  The forgetful functor $U^\bot \colon \cat{PQ \rightarrow Quant}$ is a fibration, with the lift of a quantale morphism $f\colon X\to Y$ and a pointed quantale $(Y, \bot)$ is $(X,\bot')$ where  $\bot':=f^*(\bot) = \bigvee\{x\in X\ |\ f(x)\leq_Y \bot\}$. This is the largest element of $X$ which allows for $f$ to weakly preserve the point.

\begin{proof}

  We start with a quantale morphism $f\colon X\to Y$ and a pointed quantale $(Y, \bot)$ is $(X,\bot')$ where  $\bot':=f^*(\bot) = \bigvee\{x\in X\ |\ f(x)\leq_Y \bot\}$. The Galois connection $f \dashv f^*$ exists because quantale morphisms preserve all joins.
  We know $f$ lifts to a $\cat{PQ}$ morphism $(X,\bot')\to(Y,\bot)$ by checking that $f(\bot') = f(f^*(\bot))\leq_Y \bot$, which follows from  the Galois connection. When checking if this is a cartesian lift, uniqueness follows automatically from $U^\bot$ being faithful. For arbitrary $(Z,\bot_Z) \in \cat{PQ}$ and $g \in \cat{Quant}(Z, X)$, such that ${g}\cdot f \in \cat{PQ}((Z,\bot_Z),(Y,\bot))$, we need to check that $g$ is also a $\cat{PQ}$ morphism $(Z,\bot_Z)\to (X,\bot')$. This means we need $g(\bot_Z)\leq_X f^*(\bot_Y)$. By the Galois connection, this is tantamount to $f(g(\bot_Z)) \leq_Y \bot_Y$, which is precisely what ${g}\cdot f$ being a $\cat{PQ}$ morphism grants us. 

\end{proof}

\subsection*{\cref{lemma:gqrefl}}

{\bf Lemma}
   By forgetting the fact that $\bot$ is dualizing, $\cat{GQ}$ is a reflective subcategory of $\cat{PQ_{cont}}$, the wide subcategory of ${\cat{DQ}}$ restricted to continuous morphisms. The reflector sends $(Q,\leq,\otimes,\bot)$ to $(Q',\vee',\otimes',\bot)$ where $Q':=\{q\in Q\ |\ q=q^{\bot\bot}\}$ with $a \vee' b:= (a\vee b)^{\bot\bot}$, $a \otimes' b := (a \otimes b)^{\bot\bot}$. The map to the reflector sends $q \mapsto q^{\bot\bot}$.

\begin{proof}
  Restricting the pointed quantale to its $(-)^{\bot\bot}$-closed elements (called {\it facts} by Girard) yields a Girard quantale~\cite{girard1987linear}.
  To check inclusion functor is valid (it clearly produces pointed quantales), we need to make sure $\cat{GQ}$ morphisms are continuous. Because we can eliminate $(-)^{\bot\bot}$, the continuity condition reduces to $f(x)=f(x)$. The subcategory is clearly faithful, but it is moreover full: any continuous pointed quantale morphism is a $\cat{GQ}$ morphism (the other constraints of $\cat{GQ}$ morphisms, namely being quantale morphisms and weakly preserving $\bot$, are satisfied by $\cat{PQ}$ morphisms). The universal property of the reflector, $F$, is shown for an arbitrary $\cat{PQ}$ morphism $f\colon \cX \to \cY$, where $\cY \in \cat{GQ}$. The universal map $\bar f\colon F(\cX)\to \cY$ is just a restriction of $f$ to its $(-)^{\bot\bot}$-closed elements. Reflectivity is that $\bar f$ is the unique $F(\cX) \to \cY$ map that makes $r_\cX\cdot \bar f = f$. This commutes if $f(x)=f(x^{\bot\bot})$, which is what $f$ being continuous provides (combined with elements of $\cY$ being $(-)^{\bot\bot}$-closed). $\bar f$ is the unique map which makes this commute because $r_\cX$ is surjective: it is completely determined by where $f$ assigns the $(-)^{\bot\bot}$-closed elements.
\end{proof}

\subsection*{\cref{lemma:adjrest}}

{\bf Lemma}  The adjunction $\pi_{\cat{PQ}}\dashv U_{\cat{PQ}}$ between $\cat{uIF}$ and $\cat{PQ}$ restricts to one between $\cat{uIF_{cont}}$ and $\cat{PQ_{cont}}$. 

\begin{proof}
  Let $m\colon \N[Q]\to Q$ be shorthand for the function $\Gamma\mapsto \bigotimes_{\gamma\in \Gamma} \gamma$. Let $Q \in \cat{PQ}$ and consider \\$\pi_{\cat{PQ}}U_{\cat{PQ}}(\cQ)=(\cP(\N[Q]))$ with a distinguished point $\bot'=m^{-1}(\bot^\downarrow)$. We first provide a simplifying form for $(-)^{\bot'}$ in $\pi_{\cat{PQ}}U_{\cat{PQ}}(\cQ)$. Let $\vec \Gamma \subseteq \N[Q]$ and  We consider when some $\Delta\in \N[Q]$ is an element of $\vec \Gamma^{\bot'}$. Unpacking the definition of $(-)^{\bot'}$ we have $\forall \Gamma \in \vec \Gamma\colon m(\Gamma)\otimes m(\Delta) \leq \bot$. Because $\otimes$ preserves joins, we have $m(\Delta)\otimes \bigvee_{\Gamma} m(\Gamma) \leq \bot$. Now we substitute the definition of $\varepsilon$: $m(\Delta)\otimes \varepsilon(\vec \Gamma) \leq \bot$. Using the definition of $(-)^\bot= - \multimap \bot$ in the quantale $Q$, we have $m(\Delta)\leq \varepsilon(\vec \Gamma)^\bot$. Therefore ${\vec\Gamma^{\bot'}=m^{-1}(\varepsilon(\vec\Gamma)^{\bot\downarrow})}$. A consequence is that $\varepsilon(\vec\Gamma^{\bot'}) = \varepsilon(\vec\Gamma)^{\bot}$, as the left term is the join the lower set of the right term (this relies on the surjectivity of $m$). This allows us to simplify $\vec\Gamma^{\bot'\bot'}=m^{-1}(\varepsilon(\vec \Gamma^{\bot'})^{\bot\downarrow}) = m^{-1}(\varepsilon(\vec \Gamma)^{\bot\bot\downarrow})=m^{-1}(\varepsilon(\vec \Gamma)^{\bot\bot})$.

  Continuity in $\cat{uIF}$ is defined as continuity of the corresponding $\cat{PQ}$ morphism, so every morphism of $\cat{uIF_{cont}}$ is sent to a $\cat{PQ_{cont}}$ morphism. For the reverse, let $f\colon X\to Y$ be a continuous pointed quantale morphism. It is unchanged by $U_{\cat{PQ}}$, so we determine that $f$ is a continuous $\cat{uIF}$ morphism by checking $\cP[\N[f]]$ is a continuous quantale morphism $\cP[\N[X]]\to\cP[\N[Y]]$. We use the third formulation of continuity from \cref{lemma:continuityequiv}: we suppose as hypothesis that $\vec \Gamma^{\bot'_X}\subseteq \vec \Delta^{\bot'_X}$ and $A \in \N[Y]$ be an arbitrary element of $\cP[\N[f]](\vec \Gamma)^{\bot'_Y}$. The first hypothesis can be simplified: it says $\forall \Theta \in \N[X]\colon m(\Theta) \leq \varepsilon_X(\vec \Gamma)^{\bot_X}\implies m(\Theta) \leq \varepsilon_X(\vec \Delta)^{\bot_X}$. Via surjectivity of $m$, this becomes $\varepsilon_X(\vec \Gamma)^{\bot_X}\leq_X  \varepsilon_X(\vec \Delta)^{\bot_X}$, and continuity of $f$ gives $(f(\varepsilon_X(\vec \Gamma)))^{\bot_Y}\leq_Y  (f(\varepsilon_X(\vec \Delta)))^{\bot_Y}$. The second hypothesis simplifies to $m(A) \leq \varepsilon_Y(\cP\N[f](\vec \Gamma))^{\bot_Y}$, and by naturality of the counit ($\varepsilon_X\cdot f=\cP[\N[f]]\cdot \varepsilon_Y$) we finish the continuity proof (i.e. showing that $m(A) \leq \varepsilon_Y(\cP\N[f](\vec \Delta))^{\bot_Y}$) via:
  
  \[
  m(A)\ \leq\ \varepsilon_Y\big(\cP[\N[f]](\vec\Gamma)\big)^{\bot_Y}
        \ =\ f\big(\varepsilon_X(\vec\Gamma)\big)^{\bot_Y}
        \ \leq\ f\big(\varepsilon_X(\vec\Delta)\big)^{\bot_Y}
        \ =\ \varepsilon_Y\big(\cP[\N[f]](\vec\Delta)\big)^{\bot_Y}.
\]

  Using the second formulation of continuity from \cref{lemma:continuityequiv}, we check $\varepsilon(\vec\Gamma^{\bot'\bot'})\leq \varepsilon(\vec\Gamma)^{\bot\bot}$. Using the above simplification, the left hand side becomes a join of the right hand side's lower set, hence $\varepsilon$ is continuous.  
  
  We check that the unit is continuous by defining $Q:= \pi_{\cat{PQ}}(X)$ checking continuity of $g:=\pi_{\cat{PQ}}(\eta_{X})$. One adjunction triangle identity means that $g\cdot \varepsilon_{Q} = \id_Q$. Plugging $g(\vec \Gamma)$ into the formula $(-)^{\bot'}=m^{-1}(\varepsilon(-)^{\bot\downarrow})$ and simplifying with the triangle identity we get $g(\vec \Gamma)^{\bot'}=m_Q^{-1}(\vec \Gamma^{\bot\downarrow})$. This depends monotonically on $\vec \Gamma^{\bot}$, so from this we can see the third formulation of continuity from \cref{lemma:continuityequiv}: $\vec \Gamma^\bot\subseteq \vec \Delta^\bot \implies g(\vec \Gamma)^{\bot'}\subseteq g(\vec \Delta)^{\bot'}$.

  Because the unit and counit of $\pi_{\cat{PQ}}\dashv U_{\cat{PQ}}$ are continuous and the continuous hom-sets of $\cat{uIF}$ and $\cat{PQ}$ map onto each other, the adjunction restricts to the wide subcategories of continuous morphisms.

\end{proof}

\subsection*{\cref{lemma:}}

{\bf Lemma}
  The doubling functor is self-adjoint: $D\dashv D$.

\begin{proof}
  Like $\cat{CMon}$, $\cat{PreOrdCMon}$ has biproducts, therefore the $D\dashv D$ adjunction follows from \\${\cat{PreOrdCMon}(P\times P,Q)\cong \cat{PreOrdCMon}(P,Q)^2}$ (which follows from $\times$ being a coproduct) and \\${\cat{PreOrdCMon}(P,Q)^2\cong \cat{PreOrdCMon}(P,Q^2)}$ (which follows from $\times$ being a product). We can use this bijection to derive the unit $x\mapsto \langle (x,0),(0,x)\rangle $ and counit $\langle (a,b),(c,d)\rangle \mapsto a\otimes d$.
\end{proof}

\subsection*{\cref{lemma:continuityequiv}}
{\bf Lemma}  The following are equivalent given a quantale morphism $f\colon (X,\leq,\otimes_X,\bot_\cX)$ $\to$ \\${(Y,\leq_Y,\otimes_Y, \bot_\cY)}$.

      \begin{enumerate}
    \item $\forall x \in X\colon f(x)^{\bot_\cY\bot_\cY} = f(x^{\bot_\cX\bot_\cX})^{\bot_\cY\bot_\cY}$
    \item $\forall x\in X\colon f(x^{\bot_\cX\bot_\cX})\leq_Y f(x)^{\bot_\cY\bot_\cY}$ 
    \item $\forall a,b\in X\colon a^{\bot_\cX}\leq_X b^{\bot_\cX} \implies f(a)^{\bot_\cY} \leq_Y f(b)^{\bot_\cY}$
  \end{enumerate}

\begin{proof}

$\mathbf{1 \implies 2}$: Because $(-)^{\bot_\cX\bot_\cX}$ is increasing, we have $x \leq_X x^{\bot_\cX\bot_\cX}$ and therefore we have $f(x) \leq_Y f(x^{\bot_\cX\bot_\cX})$. Because $(-)^{\bot_\cY\bot_\cY}$ is monotone, we have $f(x)^{\bot_\cY\bot_\cY}\leq_Y f(x^{\bot_\cX\bot_\cX})^{\bot_\cY\bot_\cY}$, and applying {\bf 1} to the right hand side yields {\bf 2}.

$\mathbf{2 \implies 1}$: Apply $(-)^{\bot_\cY\bot_\cY}$ to both sides. Because it is idempotent, that yields that $f(x)^{\bot_\cY\bot_\cY} \geq_Y f(x^{\bot_\cX\bot_\cX})^{\bot_\cY\bot_\cY}$. The $\leq_Y$ direction follows from $(-)^{\bot_\cX\bot_\cX}$ being increasing.

$\mathbf{1 \implies 3}$: Suppose $a^{\bot_X}\leq_X b^{\bot_X}$. Applying antitone $(-)^{\bot_X}$ gives
$b^{\bot_X\bot_X}\leq_X a^{\bot_X\bot_X}$. Applying monotone $f$ then $(-)^{\bot_Y\bot_Y}$ gives
$f(b^{\bot_X\bot_X})^{\bot_Y\bot_Y}\leq_Y f(a^{\bot_X\bot_X})^{\bot_Y\bot_Y}$, which by {\bf 1} is
$f(b)^{\bot_Y\bot_Y}\leq_Y f(a)^{\bot_Y\bot_Y}$. Applying antitone $(-)^{\bot_Y}$ and $(-)^{\bot\bot\bot}=(-)^{\bot}$
yields $f(a)^{\bot_Y}\leq_Y f(b)^{\bot_Y}$.

$\mathbf{3 \implies 2}$: Instantiating {\bf 3} at $a:=x$, $b:=x^{\bot_X\bot_X}$, the antecedent
$x^{\bot_X}\leq x^{\bot_X\bot_X\bot_X}$ is true. The consequent is $f(x)^{\bot_Y}\leq_Y f(x^{\bot_X\bot_X})^{\bot_Y}$.
Applying antitone $(-)^{\bot_Y}$ gives $f(x^{\bot_X\bot_X})^{\bot_Y\bot_Y}\leq_Y f(x)^{\bot_Y\bot_Y}$, and combining
with $f(x^{\bot_X\bot_X})\leq_Y f(x^{\bot_X\bot_X})^{\bot_Y\bot_Y}$ gives {\bf 2}.

\end{proof}

\subsection*{\cref{lemma:reflrestrict}}

{\bf Lemma}
Let $U_\pm\colon \cat{GQ \to IF^r_{\cat{cont}}}$ send a Girard quantale $\cQ$ to the reflexive subframe of $U^{\otimes\neg\oplus}(\cQ)$ and send a $\cat{GQ}$ morphism $f$ to the restriction of $U^{\otimes\neg\oplus} f$ to reflexive elements. This is functorial and right adjoint to $F_\pm:=\iota^{\rm r}\cdot F^{\otimes\neg\oplus}$.

\begin{proof}
  Consider a $\cat{GQ}$ morphism $f\colon \cQ_\cX\to \cQ_\cY$, a function $X\to Y$ between the underlying sets. We first check $U_\pm f$ is well-defined as a function. Let $R_\cX\subseteq X^2$ be the reflexive elements of $U^{\otimes\neg\oplus}(\cX)$, i.e. $\{(a,b) \in X^2\ |\ a\otimes b\leq \bot_\cX\}$. Then $f(a)\otimes f(b)\leq f(\bot_\cX) \leq \bot_\cY$ (using monotonicity and weak $\bot$ preservation of $f$ as a $\cat{GQ}$ morphism), so $U^{\otimes\neg\oplus} f(a,b) = (f(a),f(b))$ is in $R_\cY$. $U_\pm f$ preserves $\bot$ because it is a restriction of $U^{\otimes\neg\oplus} f$, which preserves $\bot$. We must check the continuity condition for all sets of positions $A,B\subseteq \N[R_\cX+R_\cX]$. We need to show, from assuming $(A^\bot \cap \N[R_\cX+R_\cX])\subseteq (B^\bot \cap \N[R_\cX+R_\cX])$, that we can prove $(f(A)^\bot\cap \N[R_\cY+R_\cY]) \subseteq (f(B)^\bot \cap \N[R_\cY+R_\cY])$. We just need to show $A^\bot\subseteq B^\bot$, since then, by continuity of $U^{\otimes\neg\oplus} f$, we have $f(A)^\bot\subseteq f(B)^\bot$, which can be restricted to the intersection with $\N[R_\cY+R_\cY]$ to prove this goal.

  Let $(\Gamma,\Delta)$ be an element $\N[X^2+X^2]$, though not necessarily in $\N[R_\cX+R_\cX]$. We can think of $\Gamma$ as a multiset of pairs (of $X$ elements) on the left and $\Delta$ as a multiset of pairs on the right. There exists a position $(\Gamma',\Delta')$ with the same $\bot$-behavior which is in $\N[R_\cX+R_\cX]$: we replace all left pairs $(\gamma_+,\gamma_-) \in X^2$ with $(\gamma_+,\gamma_+^\bot) \in R_\cX$ and all the right pairs $(\delta_+,\delta_-) \in X^2$ with $(\delta_-^\bot,\delta_-) \in R_\cX$. This has the same $\bot$ behavior because checking if $(\langle \texttt{a}_+, \texttt{a}_-\rangle,\langle \texttt{b}_+, \texttt{b}_-\rangle) \in \bot$ only depends on $\texttt{a}_+$ and $\texttt{b}_-$. This allows us to derive $A^\bot\subseteq B^\bot$:

  \begin{align*}
    (\Gamma,\Delta) &\in A^\bot &&\\
    (\Gamma',\Delta') &\in A^\bot \cap \N[R_\cX+R_\cX] && (\Gamma',\Delta')\text{ has same incompatibilities}\\
    (\Gamma',\Delta') &\in B^\bot \cap \N[R_\cX+R_\cX] &&(A^\bot \cap \N[R_\cX+R_\cX])\subseteq (B^\bot \cap \N[R_\cX+R_\cX])\text{ by hypothesis}\\
    (\Gamma,\Delta) &\in B^\bot && (\Gamma,\Delta)\text{ has same incompatibilities}\\
  \end{align*}
  
  Having shown that restricting the domain and codomain of morphisms $U^{\otimes\neg\oplus} f$ in $\cat{IF_{cont}}$ are valid $\cat{IF^r_{cont}}$ morphisms, functoriality follows from functoriality of $U^{\otimes\neg\oplus}$.  Next we show that $U_\pm$ is right adjoint to $F_\pm$. There is a natural bijection:
  
  $$\cat{GQ}(F_\pm \cX,\ \cQ) = \cat{GQ}(F^{\otimes\neg\oplus}(\iota^{\rm r}\cX),\ \cQ)\cong \cat{IF_{cont}}(\iota^{\rm r}\cX,\ U^{\otimes\neg\oplus} \cQ)\cong \cat{IF^r_{cont}}(\cX,\ U_\pm \cQ)$$

  The first equality comes from $F_\pm=\iota^{\rm r}\cdot F^{\otimes\neg\oplus}$, and the first bijection is from  $F^{\otimes\neg\oplus}\dashv U^{\otimes\neg\oplus}$. We establish the third natural bijection by first noting  there is a natural transformation $\alpha\colon \iota^{\rm r}U_\pm\Rightarrow U^{\otimes\neg\oplus}$ whose components are the inclusion of reflexive elements into a general implication frame. Showing this inclusion function is continuous is tantamount to showing for $A,B\subseteq \N[R_\cX+R_\cX]$ and assuming $(A^\bot \cap \N[R_\cX+R_\cX])\subseteq (B^\bot \cap \N[R_\cX+R_\cX])$, that  $A^\bot \subseteq B^\bot$. This follows from the same reasoning as above where every non-reflexive position has a corresponding reflexive position with the same $\bot$-behavior. 
  
  We can postcompose with these components to obtain a function with domain $\cat{IF^r_{cont}}(\cX,\ U_\pm\cQ)$ and codomain $\cat{IF_{cont}}(\iota^{\rm r}\cX, U^{\otimes\neg\oplus} \cQ)$. In the other direction, we corestrict the function to reflexive elements. Thus our purported bijection is $\alpha_\cQ \circ -$ and $(-)|_{R}$. These are inverse because corestriction and inclusion do not change how functions act on elements of the domain.
\end{proof}

\subsection*{\cref{lemma:rsrequiv}}

{\bf Lemma}
  The elements of the Girard quantale $F^{\otimes\neg\oplus}(\cX)$ are `roles' $\R$ (\cref{def:rolcont}), and the $\RSR$ function (\cref{def:rsr}) restricted to roles is precisely the $(-)^\bot$ operation of the quantale.  

\begin{proof}
  An element of $\R$ is an element $R \in \cP[\N[X+X]]$ for which $R=\RSR(\RSR(R))$. An element $A\in \hat \cX$ is a subset $A\subseteq \N[X+X]$ such that, $A^{\bot\bot}=A$.
  
  First we must check that every element of $\R$ is a lower set. Suppose $x \leq y$ and $y \in R$ for some $R \in \R$. Because $\R=\im(\RSR)$, there exists some $S \subseteq \N[X+X]$ such that $R=\RSR(S)$. Then we infer that $\forall s \in S\colon s+y \in \bot$. However, $x \leq y$ means that anything which sums with $y$ to be in $\bot$ must sum with $x$ to be in $\bot$, therefore $\forall s \in S\colon s+x \in \bot$, which implies $x \in R$. 
  
  So elements of $\R$ and $\hat \cX$ are both lower sets, and we just need to verify that the property of being a fix point of $\RSR^2$ is equivalent to being a fix point of $(-)^{\bot\bot}$. This is true because, as operations on lower sets, they are identical: $\RSR(X)=X^\bot=\{y\in \N[X+X]\ |\ \forall x \in X \colon x+y \in \bot \}$.
\end{proof}

\subsection*{\cref{lemma:twistrestrict}}

{\bf Lemma}
  Let $\hat \cX := U_\pm(\cQ)$ for some Girard quantale $\cQ$.
  These elements (which are pairs of elements of $\cQ$ that satisfy reflexivity) are closed under the quantale operations of the {\it twisted quantale} $\cQ\times\cQ^{\rm op}$, i.e. $\otimes\times \parr$,\footnote{Note $x\parr y := (x^\bot\otimes y^\bot)^\bot$, and this operation is a monoidal product for the opposite Girard quantale.} and $\vee \times \wedge$.

\begin{proof}

  Swapping elements is invariant for property of $\texttt{a}_+\otimes \texttt{a}_-\leq \bot$ because $\otimes$ is commutative.
  For $\texttt{A} (\otimes \times \parr) \texttt{B}$ to be closed we need:
\begin{align*}
  \langle \texttt{a}_+,\texttt{a}_-\rangle &(\otimes \times \parr) \langle \texttt{b}_+,\texttt{b}_-\rangle \in \hat \cX \\ 
  \langle \texttt{a}_+ \otimes \texttt{b}_+&,\ \texttt{a}_-\parr \texttt{b}_-\rangle \in \hat \cX  &&\text{Defn of }(\otimes \times \parr)\\ 
  \texttt{a}_+ \otimes \texttt{b}_+&\leq (\texttt{a}_-\parr \texttt{b}_-)^\bot &&\text{Elementhood of }\hat \cX\\ 
  \texttt{a}_+ \otimes \texttt{b}_+&\leq \texttt{a}_-^\bot\otimes \texttt{b}_-^\bot &&\text{Defn of }\parr\\ 
\end{align*}

This follows from $\otimes$ being monotonic and $\texttt{A}, \texttt{B} \in \hat \cX$. Now we check if $\texttt{A}(\vee \times \wedge)\texttt{B}$ is closed.

\begin{align*}
  \langle \texttt{a}_+,\texttt{a}_-\rangle &(\vee \times \wedge) \langle \texttt{b}_+,\texttt{b}_-\rangle \in \hat \cX \\ 
    \langle \texttt{a}_+ \vee \texttt{b}_+&,\ \texttt{a}_-\wedge \texttt{b}_-\rangle \in \hat \cX  &&\text{Defn of }(\vee \times \wedge)\\ 
    \texttt{a}_+ \vee \texttt{b}_+&\leq (\texttt{a}_-\wedge \texttt{b}_-)^\bot &&\text{Elementhood of }\hat \cX\\ 
    \texttt{a}_+ \vee \texttt{b}_+&\leq \texttt{a}_-^\bot\vee \texttt{b}_-^\bot &&\text{Property of }(-)^\bot,\vee\\ 
\end{align*}

This follows from $\vee$ being monotonic and $\texttt{A}, \texttt{B} \in \hat \cX$. 

\end{proof}

\subsection*{\cref{lemma:undercont}}

{\bf Lemma}
$U_\pm^{\rm c}$ and $F_\pm^{\rm c}$ are well-defined as functors, with $F_\pm^{\rm c}\dashv U_\pm^{\rm c}$.

\begin{proof}
  We first show that $\iota^{\rm c}\cdot F_\pm$ corestricts to $\cat{GQ^{ji}}$. Every principal lower set is idempotent in virtue of the starting frame being idempotent, and every other element is a join of principal lower sets.
  
  $U_\pm^{\rm c}$ is a functor. We first check $U_\pm^{\rm c}$ sends $\cat{GQ^{ji}}$ morphisms $f\colon \cQ\to \cQ'$ to $\cat{IF^c_{cont}}$ morphisms, which requires confirming that IC (idempotent and containment-satisfying) elements $\cI$ of $U^{\rm c}_\pm \cQ$ are sent to IC elements of $U^{\rm c}_\pm \cQ'$ by $U^{\rm c}_\pm f = f \times f$. Firstly, because $f$ is a monoid homomorphism, it must send idempotent elements to idempotent elements. For any join-idempotent Girard quantale $\cQ$, an element $\langle p,q\rangle$ of $U_\pm^{\rm c} \cQ$ satisfies containment if $p \otimes q \otimes r \leq \bot$ for all $r \in \cQ$, which is tantamount to $p\otimes q = 0$ (where $0$ is the lattice bottom). So we must check that $f(p)\otimes f(q) = 0_\cQ'$, which holds because $0$ (the empty join) is preserved by $f$. Once again, $\bot$ is preserved by $U_\pm^{\rm c} f$ because it is the restriction of $U_\pm f$, which preserves $\bot$. 

  Let $C\subseteq X^2$ be the set of IC elements in $U^{\rm c}_\pm \cQ$ for $\cQ:=(X,\otimes,\vee, \bot_\cQ)$. Much like the proof of \cref{lemma:reflrestrict}, we show continuity by using the continuity of $U_\pm f$. This requires us to prove, for any  $A,B\subseteq \N[C+C]$ (with an arbitrary element of $A$ being $\langle \Gamma,\Delta\rangle$ with $\Gamma,\Delta \in \N[C]$) and assuming that $(A^\bot \cap \N[C+C])\subseteq (B^\bot\cap \N[C+C])$, that $A^\bot\subseteq B^\bot$. For any multiset $\Xi \in \N[X^2]$ of pairs of elements in $\cQ$ in $A$, let $p_\Xi:=\bigotimes_{(\xi_+,\xi_-)\in \Xi}\xi_+$ and  $n_\Xi:=\bigotimes_{(\xi_+,\xi_-)\in \Xi}\xi_-$. 
  
  Now, suppose $(\Theta,\Omega) \in A^\bot$, i.e. $\forall \langle \Gamma,\Delta\rangle \in A\colon p_\Gamma \otimes p_\Theta \otimes n_\Delta \otimes n_\Omega \leq \bot_\cQ$. Because $\cQ$ is a join-idempotent quantale, we have $p_\Theta=\bigvee_{e \in \cI\cap p_\Theta^\downarrow} e$ and $n_\Omega=\bigvee_{d \in \cI\cap n_\Omega^\downarrow} d$. Therefore, our hypothesis is that $\bigvee_{e \in \cI\cap p_\Theta^\downarrow}\bigvee_{d \in \cI\cap n_\Omega^\downarrow} e \otimes d \otimes p_\Gamma \otimes n_\Delta \leq \bot_\cQ$ and we need to prove, for an arbitrary $\langle \Gamma',\Delta'\rangle\in B$ that 
  
  $$\bigvee_{e \in \cI\cap p_{\Theta}^\downarrow}\bigvee_{d \in \cI\cap n_{\Omega}^\downarrow} e \otimes d \otimes p_{\Gamma'} \otimes n_{\Delta'} \leq \bot_\cQ$$ 
  
  We can show this latter inequality by showing it for an arbitrary choice of $e \in \cI\cap p_{\Theta}^\downarrow$ and $d \in \cI\cap n_{\Omega}^\downarrow$. 
  
  Consider $\mathtt{x} = \langle \{(e,0)\}, \{(0,d)\} \rangle$, which is in $ \N[C+C]$ because $e$,$d$, and $0$ are all idempotent elements of $\cQ$ and because $0$ is an absorbing element, so $e\otimes 0=0\otimes d=0$. 
  
  \begin{align*}
    e\otimes d &\leq p_\Theta \otimes n_\Omega && e \leq p_\Theta \wedge d\leq n_\Omega\\
    p_\Gamma \otimes e \otimes n_\Delta \otimes d &\leq p_\Gamma\otimes p_\Theta \otimes n_\Delta \otimes n_\Omega && - \otimes (p_\Gamma \otimes n_\Delta)  \text{ is monotone}\\ 
    p_\Gamma \otimes e \otimes n_\Delta \otimes d &\leq \bot_\cQ && (\Theta,\Omega) \in A^\bot\\
    \mathtt{x} &\in A^\bot \\
    \mathtt{x} &\in B^\bot && (A^\bot \cap \N[C+C])\subseteq (B^\bot\cap \N[C+C])\\
    \forall \langle \Gamma',\Delta'\rangle \colon e \otimes d \otimes p_{\Gamma'} \otimes n_{\Delta'} &\leq \bot_\cQ \\
  \end{align*}
  
  Therefore this inequality holds for the join of all such $e,d$, which was needed to show $(\Theta,\Omega) \in B^\bot$.
  
  The argument for extending the adjunction to $\cat{IF^c_\pm}$ follows \cref{lemma:reflrestrict}. There is a natural bijection $\cat{IF^r_{cont}}(\iota^{\rm c}\cX,U_\pm \cQ) \cong \cat{IF^c_{cont}}(\cX,U^{\rm c}_\pm \cQ)$ given by restriction to IC elements and the inclusion of IC elements.  
\end{proof}

\subsection*{\cref{lemma:idemsubquant}}

{\bf Lemma}
  Let $\cQ:=(Q,\otimes,I,\vee)$ be any quantale. There is an idempotent quantale $(Q_\otimes,\otimes,\tilde\vee)$ with elements $Q_\otimes:=\{q \in Q\ |\ q \otimes q = q\}$, the same $\otimes$ as $\cQ$, and (binary) joins are given by $x\ \tilde \vee\ y := x \vee y \vee (x\otimes y)$.\footnote{This is a special case of general joins: $\tilde \bigvee_{i\in I} x_i := \bigvee_{\varnothing \subset J \subseteq_{\rm fin} I} \bigotimes_{j\in J} x_j$.}

\begin{proof}
  
By construction, every element of $Q_\otimes$ is idempotent. We next need to show that joins are given by $\tilde \bigvee_{i\in I} x_i := \bigvee_{\varnothing \subset J \subseteq_{\rm fin} I} \bigotimes_{j\in J} x_j$. First we must show this is idempotent:

\begin{align*}
  (\bigvee_{\varnothing \subset J \subseteq_{\rm fin} I} \bigotimes_{j\in J} x_j) \otimes (\bigvee_{\varnothing \subset K \subseteq_{\rm fin} I} \bigotimes_{k\in K} x_k) \\
  \bigvee_{\varnothing \subset J,K \subseteq_{\rm fin} I} \bigotimes_{j\in J} x_j \otimes  \bigotimes_{k\in K} x_k &&\otimes \text{ distributes over joins}\\
  \bigvee_{\varnothing \subset J \subseteq_{\rm fin} I} \bigotimes_{j\in J} x_j &&\text{See below}\\
\end{align*}

The last step is sound because for any choice of $K$, multiplying the product $\bigotimes_j x_j$ simply produces $\bigotimes_{j'\in J'}x_{j'}$ for some different $J'$ in between $\varnothing$ and $I$. 

To show this is an upper bound of $x_i$ for $i \in I$, it is a join of all of the singleton sets $\{x_i\}$ and is therefore above them all in $Q$. To see it is the {\it least} upper bound when restricting to idempotent elements, let $c$ be an idempotent element above all $x_i$. Using the monotonicity of $\otimes$ and $\leq$ we have, for any nonempty $J \subseteq I$,  $\bigotimes_{j \in J} x_j \leq \bigotimes_{j \in J} c$, and the idempotency of $c$ means $\bigotimes_{j \in J} x_j \leq c$. Therefore, because all elements in the join that constitutes $\tilde \bigvee_{i\in I} x_i$ are below $c$ in $\cQ$, we have $\tilde \bigvee_{i\in I} x_i \leq c$. Lastly, $\otimes$ distributes over this join because it is a join in $Q$. Note for the $I=\{1,2\}$ case, this formula evaluates to $x\ \tilde \vee\ y := x \vee y \vee (x\otimes y)$. 

\end{proof}

\subsection*{\cref{lemma:mixrestrict}}

{\bf Lemma}
  The elements of $\hat \cX = U^{\rm c}_\pm(\cQ)$ for some Girard quantale $\cQ$ are closed under the quantale operations of the {\it mixed monoid} $\cQ_\otimes^\otimes \times \cQ_\otimes^\vee$, whose components $\cQ_\otimes^\otimes$ and $\cQ_\otimes^\vee$ respectively are the underlying $\otimes$ and join monoids of $\cQ_\otimes$. 

\begin{proof}
  Let $\hat \cX = U^{\rm c}_\pm(\cQ)$ for some Girard quantale $\cQ$. For $A (\otimes \times \tilde\vee) B$ to be closed (i.e. preserve containment satisfaction), we need $ \langle \texttt{a}_+,\texttt{a}_-\rangle (\otimes \times \tilde \vee)  \langle \texttt{b}_+,\texttt{b}_-\rangle$ to satisfy containment, i.e. $(\texttt{a}_+ \otimes  \texttt{b}_+) \otimes (\texttt{a}_-\ \tilde \vee\ \texttt{b}_-)$ to equal $0_\cQ$:
  
  \begin{align*}
     (\texttt{a}_+ \otimes  \texttt{b}_+) &\otimes (\texttt{a}_-\ \tilde \vee\ \texttt{b}_-) &&\text{} \\
     (\texttt{a}_+ \otimes  \texttt{b}_+) &\otimes (\texttt{a}_-  \vee \texttt{b}_- \vee (\texttt{a}_-\otimes \texttt{b}_-)) && \text{Defn of }\tilde\vee \\
     (\texttt{a}_+ \otimes  \texttt{b}_+ \otimes \texttt{a}_-) \vee (\texttt{a}_+ \otimes  \texttt{b}_+ \otimes\texttt{b}_-) &\vee (\texttt{a}_+ \otimes  \texttt{b}_+ \otimes\texttt{a}_-\otimes\texttt{b}_-) &&\otimes \text{ distributes over joins} \\
     (0_\cQ \otimes  \texttt{b}_+) \vee (\texttt{a}_+ \otimes  0_\cQ) &\vee (0_\cQ \otimes 0_\cQ) &&\langle \texttt{a}_+,\texttt{a}_-\rangle, \langle \texttt{b}_+,\texttt{b}_-\rangle\text{ satisfy containment} \\
     0_\cQ \vee 0_\cQ &\vee 0_\cQ &&0_\cQ\text{ is absorbing} \\
     &0_\cQ  &&\vee \text{ is idempotent} \\
  \end{align*}
  
\end{proof}

\end{document}